%% file: main_preprint.tex
\documentclass[12pt]{article}

\usepackage{arxiv}

\title{High-Probability Minimax Adaptive Estimation in Besov Spaces via Online-to-Batch}
\usepackage{times}
\input{mypreambule}
\usepackage{natbib}
\usepackage[algoruled,linesnumbered, vlined]{algorithm2e}
\usepackage{amsthm}
\newtheorem{theorem}{Theorem}
\newtheorem{corollary}{Corollary}
\newtheorem{assumption}{Assumption}
\newtheorem{lemma}{Lemma}
\newtheorem{definition}{Definition}

\DeclareMathSizes{10}{9}{6}{5}
\author{%
  Paul Liautaud \\
  Sorbonne Université, CNRS, LPSM \\
  F-75005 Paris, France \\
  \texttt{paul.liautaud@sorbonne-universite.fr} \\
  \And
  Pierre Gaillard \\
  Université Grenoble Alpes, Inria \\
  CNRS, Grenoble INP, LJK \\
  38000 Grenoble, France \\
  \texttt{pierre.gaillard@inria.fr} \\
  \And
  Olivier Wintenberger \\
  Sorbonne Université, CNRS, LPSM \\
  F-75005 Paris, France \\
  \texttt{olivier.wintenberger@sorbonne-universite.fr}
}

\begin{document}

\maketitle

\begin{abstract}%An attempt

We study nonparametric regression over Besov spaces from noisy observations under sub-exponential noise, aiming to achieve minimax-optimal guarantees on the integrated squared error that hold with high probability and adapt to the unknown noise level. To this end, we propose a wavelet-based online learning algorithm that dynamically adjusts to the observed gradient noise by adaptively clipping it at an appropriate level, eliminating the need to tune parameters such as the noise variance or gradient bounds. As a by-product of our analysis, we derive high-probability adaptive regret bounds that scale with the $\ell_1$-norm of the competitor.
Finally, in the batch statistical setting, we obtain adaptive and minimax-optimal estimation rates for Besov spaces via a refined online-to-batch conversion. This approach carefully exploits the structure of the squared loss in combination with self-normalized concentration inequalities.
\end{abstract}


\section{Introduction}

A fundamental challenge in statistics and machine learning is the estimation of an unknown function $f$ from noisy observations. Given a sample of $T \ge 1$ i.i.d.\ data points $(X_t, Y_t)_{t=1}^T$ following the regression model
\[
Y_t = f(X_t) + \varepsilon_t, \quad t = 1, \dots, T,
\]
where $\varepsilon_t$ denotes a sub-exponential zero-mean noise, the goal is to construct an estimator $\bar f_T$ of the true function $f : \mathcal X \to \mathbb R$, supported on a compact domain $\mathcal X \subset \mathbb R^d$, $d \ge 1$. The quality of an estimator $\bar f_T$ is measured by the integrated squared error $\|\bar f_T - f\|_{L^2(\mathbb P_X)}$. When $f$ is assumed to belong to a Besov space $\mathcal B^s_{pq}$ (see Section~\ref{section:risk}), the landmark work of \cite{donoho1998minimax} established the minimax-optimal rates for the mean integrated squared error, which are attained by wavelet shrinkage methods.

While this classical theory focuses on optimality in expectation, many modern learning applications—ranging from safety-critical systems to precision medicine—require guarantees that hold with \emph{high probability}. This motivates the study of estimators that achieve minimax-optimal performance beyond the average-case regime. Our objective is therefore to design an estimator $\bar f_T$ that satisfies, with high probability, a minimax-adaptive risk bound of the form
\begin{equation}
\label{eq:bound_intro}
\|\bar f_T - f\|_2^2 = O\bigg( \|f\|_{\mathcal B}^{\frac{2d}{2s+d}} (\sigma^2)^{\frac{2s}{2s+d}} T^{-\frac{2s}{2s + d}} \bigg)\, ,
\end{equation}
for any continuous function $f \in \mathcal B^s_{pq}$, without prior knowledge of the Besov parameters $(s,p,q)$, the Besov norm $\|f\|_{\mathcal B}$, or the noise variance $\sigma^2$. This bound is said to be \emph{noise-level–aware} since it exhibits minimax-optimal dependence on  the noise level $\sigma$. Importantly, unlike classical minimax analyses \citep[e.g.,][]{delyon1996minimax, donoho1998minimax}, which characterize rates asymptotically as $T \to \infty$ under known or fixed noise level, our result adapts to the variance $\sigma^2$ and matches the noise-level-aware minimax rates recently derived in \cite{devore2025optimal}. Achieving such guarantees simultaneously with high probability and full adaptivity remains challenging for standard wavelet shrinkage and aggregation techniques.

To address this challenge, we turn to the framework of \emph{online learning}. Online algorithms are natural candidates for adaptive estimation, as they are designed to control regret without prior knowledge of the data-generating process. By defining the estimator $\bar{f}_T$ as the average of the online iterates, $\bar{f}_T = \frac{1}{T} \sum_{t=1}^T \hat{f}_t$, one can relate the \emph{excess risk} to the \emph{stochastic regret} via Jensen's inequality:
\begin{align}
    \|\bar{f}_T - f\|_2^2 &= \mathbb{E}_{XY}[(\bar{f}_T(X) - Y)^2] - \mathbb{E}_{XY}[(f(X) - Y)^2] \notag \\
    &\leq \frac{1}{T} \sum_{t=1}^T \mathbb{E}_{t-1}[(\hat{f}_t(X_t) - Y_t)^2] - \mathbb{E}_{t-1}[(f(X_t) - Y_t)^2]\, , \label{eq:regret_sto}
\end{align}
where $\mathbb{E}_{t-1}$ denotes the expectation conditioned on the first $t-1$ samples and since $(X_t,Y_t),t\ge1$ are i.i.d. 
This \emph{online-to-batch conversion} principle \citep{littlestone1989line, cesa2004generalization,shalev2025online} ensures that any regret bound for convex losses yields an equivalent excess risk bound \emph{in expectation}. However, despite its popularity, standard online-to-batch conversion typically fails to provide sharp \emph{high-probability} risk guarantees and adaptations of the online algorithms are required \citep{wintenberger2017optimal,van2023high}. While the expected regret may be optimal, fluctuations along the regret trajectory can lead to suboptimal concentration behavior when translated to the batch setting. This issue is particularly severe in nonparametric regimes, where the minimax rates can be as fast as $T^{-\nicefrac{2s}{(2s+d)}}$, which is strictly faster than $T^{-1/2}$ when $s \ge d/2$. In this regime, standard first-order concentration inequalities are insufficient to control these fluctuations at the optimal rate. Recently, \citet{van2023high} proposed a general reduction for deriving high-probability risk bounds from online learning guarantees. However, this approach is not sufficient to attain minimax-optimal rates in the low-regularity regime $s < \tfrac{d}{2}$ where the best known online regret bounds scale as $O(T^{-\nicefrac{s}{d}})$ \citep{rakhlin2014online}, which is strictly slower than the optimal batch risk rate $O(T^{-\nicefrac{2s}{(2s+d)}})$.
The central challenge in this paper is therefore to design online algorithms $(\hat f_t)$ whose adaptive nonparametric regret guarantees in \eqref{eq:regret_sto} can be converted into minimax-optimal \emph{high-probability} risk bounds of the form \eqref{eq:bound_intro}.

Recently, \cite{liautaud2025minimax} introduced a wavelet-based online algorithm that achieves adaptive minimax-optimal regret in the adversarial setting. Their approach relies in particular on so-called \emph{comparator-adaptive} online algorithms, which guarantee regret bounds of order $\tilde O(\|\c\|_1 G \sqrt{T})$ for $G$-Lipschitz losses against any comparator $\c \in \R^N, N\ge 1$. However, in the unbounded stochastic setting considered here, the loss functions are no longer Lipschitz, and such guarantees can fail dramatically. To overcome this difficulty, we employ a \emph{gradient clipping} strategy \citep{zhang2022parameter} with an adaptive online threshold. This modification allows us to recover high-probability control of the regret while preserving comparator adaptivity.
As a by-product of our analysis, we obtain new results for stochastic online convex optimization, establishing regret bounds of the form $\tilde O(\|\c\|_1 (G + \sigma)\sqrt{T})$
holding with high probability for any comparator $\c \in \mathbb R^N$ with $\|\c\|_\infty < \infty$ and with unknown noise level $\sigma$, answering an open question in \cite{zhang2022parameter}.

\paragraph{Notations.}
The notation $O(\cdot)$ hides universal constant factors, while $\tilde O(\cdot)$ additionally hides polylogarithmic factors. For any integer $k \ge 1$, we denote $[k] := \{1,\dots,k\}$. Boldface letters are used to denote multivariate quantities, whereas standard (non-bold) letters denote scalar quantities. We denote $a \vee b = \max\{a,b\}$ and, for any $x \in \R$ and threshold $\tau > 0$, the clipping operator of $x$ in $[-\tau,\tau]$ is defined by $\operatorname{clip}(x,\tau) = x$ if $|x| \le \tau$ and $\operatorname{sign}(x)\tau$ otherwise.

\paragraph{Contributions and outline of the paper.}
In Section~\ref{section:parameter-free-high-prob}, we first derive by-product results for stochastic online convex optimization with adaptive guarantees. We establish high-probability comparator-adaptive regret bounds of order $\tilde O(\|\c\|_1\sqrt{T})$ in the constrained setting $\|\c\|_\infty \le C$, under sub-exponential gradient noise. In Section~\ref{section:adaptive_clipping}, we extend these results to the unknown-noise setting by introducing an adaptive gradient clipping procedure. This method removes the need for prior knowledge of the noise level and is general enough to be applicable in other online learning contexts.
Finally, we leverage these online learning guarantees to design an adaptive nonparametric estimator in Section~\ref{section:risk}. The resulting algorithm achieves minimax, noise-level-aware risk bounds with high probability, without requiring prior knowledge of the noise $\sigma$.

\subsection{Related work}

\paragraph{Nonparametric regression over Besov spaces.}
The minimax theory for nonparametric regression over Besov spaces is well established, with classical results characterizing optimal rates in expectation; see, e.g., \citet{delyon1996minimax,donoho1998minimax}. More recently, \citet{devore2025optimal} derived sharp upper and lower bounds that are explicitly \emph{noise-level–aware} in the batch setting, under the assumption of known Gaussian noise variance. These results highlight the importance of adaptivity to the noise level and motivate the use of online learning techniques to achieve such adaptivity in a data-driven manner.
Recently \citet{liautaud2025minimax} proposed a polynomial-time online algorithm achieving minimax-optimal regret against adversarial sequences and Besov-smooth competitors. While their focus is on adversarial regret, our work leverages related online ideas to derive high-probability, noise-level–adaptive risk guarantees in the stochastic regression setting.

\paragraph{Online-to-Batch.}
Classical online-to-batch reductions \citep{littlestone1989line, cesa2004generalization} relate online regret to batch risk, but are limited to guarantees in expectation and do not yield fast rates with high probability. The route to obtain high probability $O(\tfrac{1}{T})$ excess risk bounds via online to batch conversions was initiated by \cite{wintenberger2017optimal}. More recently, \citet{van2023high} proposed a general reduction for deriving high-probability risk bounds from sequential regret. However, this approach does not recover minimax-optimal rates in our setting, in particular in the low-regularity regime $s < \tfrac{d}{2}$, where the best known online regret bounds scale as $O(T^{-s/d})$ \citep{rakhlin2014online}, which is strictly slower than the optimal batch rate $O(T^{-2s/(2s+d)})$.
In contrast, our analysis departs from generic online-to-batch reductions and relies on a refined control of the stochastic regret in~\eqref{eq:regret_sto}, under a stochastic directional derivative condition (Assumption~\ref{assump:exp-concave}). This approach is inspired by recent work on noisy online convex optimization \citep{wintenberger2024stochastic}, which established high-probability fast rates under sub-Gaussian perturbations, but extends it to a more general setting tailored to nonparametric regression and adaptive wavelet estimation.

\paragraph{Online convex optimization with stochastic unbounded gradients.}
Most online convex optimization algorithms require a priori gradient bounds, or incur exponential penalties when gradients are unbounded \citep{cutkosky2017online}. Recent works address unbounded gradients through stochastic assumptions and gradient clipping. In particular, \citet{jun2019parameter} obtain in-expectation guarantees under sub-exponential noise, while \citet{zhang2022parameter} derive high-probability, parameter-free regret bounds using regularization techniques. As a by-product of our analysis, we complement this line of work by establishing high-probability $\ell_1$ comparator-adaptive regret guarantees in a setting with bounded iterates and unbounded stochastic gradients.

\section{High probability comparator-adaptive regret with constrained iterates and noisy unbounded gradients}
\label{section:parameter-free-high-prob}

We consider the following setting of stochastic online prediction. Let $N \ge 1$ and $(\ell_t : \R^N \to \R_+)_{t\ge 1}$ a collection of random convex differentiable loss functions, sequentially observed. At each time $t$, a learner forms a prediction $\c_t \in \R^N$ based on the past observations $\mathcal F_{t-1} = \{\c_1,\nabla \ell_1(\c_1),\dots, \c_{t-1},\nabla \ell_{t-1}(\c_{t-1})\}$, where $\nabla \ell_s(\c_s)$ is the gradient of $\ell_s$ at $\c_s$. The learner aims at minimizing the cumulative stochastic regret \[
R_T(\c) := \sumT \E_{t-1}[\ell_t(\c_t)] - \E_{t-1}[\ell_t(\c)], \qquad \text{where} \quad \E_{t-1}[\cdot] = \E[\cdot \mid \mathcal F_{t-1}],
\]
with respect to all $\c \in \R^N$ such that $\|\c\|_\infty \leq C$ for some fixed $C > 0$. In particular, considering the Dirac masses, one obtains $\ell_t = \E_{t-1}[\ell_t]$ and the stochastic regret matches the regret more commonly used in the online (adversarial) learning literature.
In this section, we present Algorithm~\ref{alg:param_free_unbounded_gradients}, which achieves a stochastic regret bound of order $\tilde O(\|\c\|_1 (G + \sigma)\sqrt{T})$
with high probability with unbounded gradients of variance $\sigma$ and expected bound $G \ge \big|\E_{t-1}[\nabla \ell_t(\c_t)]\big|$.
In particular, we assume that gradients' noise admit sub-exponential tails, as formalized in Assumption~\ref{assump:subexp}.

% \paragraph{Gradient noise assumption.}

%  We consider that $\bxi_t = \nabla \ell_t(\c_t) - \E_{t-1}[\nabla \ell_t(\c_t)]$ is a sub-exponential noise, whose properties are detailed in the following Assumption~\ref{assump:subexp}.

\begin{assumption}[Sub-exponential gradient noise]
\label{assump:subexp}
The gradient sequence $(\hat \g_t)_{t \ge 1}$ is said to have conditionally sub-exponential noise if there exist constants $\nu,\mu > 0$ such that, for all $t \ge 1$, conditionally on the past $\mathcal F_{t-1}$, the centered gradient noise
\[
\bxi_t := \hat \g_t - \E_{t-1}[\hat \g_t]
\]
satisfies the following tail bound: for all $u \ge 0$,
\[
\max_{1 \le n \le N}
\mathbb P_{t-1}\!\left(
|\xi_{n,t}| \ge u
\right)
\le
\exp\!\left(
-\frac{1}{2}
\min\!\left(
\frac{u^2}{\nu^2},
\frac{u}{\mu}
\right)
\right),
\]
where $\mathbb P_{t-1}$ denotes the conditional probability given the past observations $\mathcal F_{t-1}$.
\end{assumption}

\paragraph{Stochastic directional derivative condition.}
 
We introduce the following assumption on the random losses $(\ell_t)$ revealed by the environment.
\begin{assumption}[Stochastic directional derivative condition] 
\label{assump:exp-concave}
We say that the (random) losses $\ell_t$ satisfy the \emph{stochastic directional derivative condition} for some constant $\alpha > 0$ if for every $\c_1,\c_2 \in [-C,C]^N, C >0$
\[
\E_{t-1}\big[\ell_t(\c_1) - \ell_t(\c_2)\big] \leq \E_{t-1}\Big[\nabla\ell_t(\c_1)^\top(\c_1 - \c_2) - \frac{\alpha}{2}(\nabla \ell_t(\c_1)(\c_1 - \c_2))^2\Big]\,.
\]
\end{assumption}
Remark that the stochastic directional derivative condition holds for non-convex losses and unbounded gradients. Setting $\alpha= 0$ coincides with the convexity of $\E_{t-1}[\ell_t], t\ge1$. This assumption is a weaker condition than exp-concavity of the losses $(\ell_t)$, and follows, in the deterministic setting, from the exp-concavity of the loss functions \citep[Lemma 4.3][]{hazan2016introduction}. It coincides with Assumption {\bf (H2)} of \cite{wintenberger2024stochastic}.

\subsection{Algorithm: Adaptive Learning with Unbounded Noisy Gradients via Clipping}
\label{section:learning_with_known_clipping}

We now introduce an online learning algorithm designed to operate in the presence of unbounded gradient observations, while retaining high-probability, comparator-adaptive regret guarantees of type $\tilde O(\|\c\|_1\sqrt{T})$. 
Our construction is modular: each coordinate is updated by an independent one-dimensional online subroutine, which receives clipped gradient feedback. This design allows us to leverage generic comparator-adaptive online algorithms (satisfying Assumption~\ref{assump:parameter-free-algo}), while ensuring robustness in regimes where gradient magnitudes may be large.

\begin{algorithm}[htbp]
\caption{Adaptive Learning with Unbounded Noisy Gradients via Clipping at time $t$%Coordinate-wise clipping with unbounded gradients and constrained iterates.
}
\label{alg:param_free_unbounded_gradients}
    \textbf{Input:} $N \ge 1$ algorithms $\mathcal A_1, \dots \mathcal A_N$ satisfying Assumption~\ref{assump:parameter-free-algo}, diameters $(C_n)$ of feasible sets, bounds on gradients $|\E_{t-1}[\nabla \ell_t(\c_t)_n]|\le G_{n,t}, n\in [N]$ and clipping margin $\Delta_t > 0$ \;
    % \textbf{Initialization:} $\c_1 = \tilde \c_1 = 0_{\R^N}$ \;
    % \For{$t = 1, \dots, T$}{
    Predict $\c_t = (c_{n,t})_{n=1}^N$ \;
    Receive noisy gradient $\hat \g_t = \nabla \ell_t(\c_t)$\;% noisy gradient of $L_t(\c_t)$ \;
    \For{$n = 1, \dots, N$}{
    Define clipping threshold as $\bar G_{n,t} = G_{n,t} + \Delta_t$\;
    Clip gradient $\bar g_{n,t} = \max(-\bar G_{n,t},\min(\bar G_{n,t},\hat g_{n,t}))$ or $\bar g_{n,t} = 0$ if $G_{n,t} = 0$\; %\footnotemark\;
    Receive $c_{n,t+1} \in [-C_n,C_n]$ from $\mathcal A_n$ using $(\bar g_{n,t}, \bar G_{n,t})$ \;
    }
    % }
    \textbf{Output:} $\c_{t+1}$
\end{algorithm}

\paragraph{Assumption on the coordinate subroutine.}
We assume that each coordinate subroutine $\mathcal A_n, n \in [N]$ satisfies the following assumption.

\begin{assumption}[Comparator-adaptive algorithm]
\label{assump:parameter-free-algo}
Let $T\ge1, C>0$ and let $g_1,\dots,g_T\in\mathbb R$ be a sequence of scalar gradients such that
$|g_t|\le G_t$ for all $t\le T$.
An online algorithm $\mathcal A$ starting at $c_1 \in [-C,C]$ and producing predictions $(c_t)_{t=1}^T \subset [-C,C]$
is said to be \emph{comparator-adaptive} if for some factors $\Xi_1,\Xi_2>0$ and for any comparator $c\in [-C,C]$,
\[
\sum_{t=1}^T g_t (c_t - c)
\;\le\;
|c-c_1|
\bigg(
\Xi_1 \,\sqrt{\textstyle \sum_{t=1}^T g_t^2}
+ \Xi_2\,\sup_{1\le t \le T} G_t \bigg).
\]
\end{assumption}

This assumption holds for a broad class of first-order online learning algorithms such as online mirror descent with self-tuned learning rates or coin-betting style updates \citep{orabona2016coin, cutkosky2018black, mhammedi2020lipschitz, jacobsen2022parameter, chen2021impossible}. These are referred to as \emph{comparator-adaptive} algorithms as they provide optimal adaptivity to the magnitude of the comparator $|c|$ in the regret, typically at the cost of logarithmic factors absorbed into $\Xi_1$ and $\Xi_2$. 
Algorithms satisfying the linear regret bound in Assumption~\ref{assump:parameter-free-algo} are primarily developed for the unconstrained setting ($C = \infty$). We here consider an extension in which the predictions $c_t$ are constrained to $[-C,C]$, typically at the cost of a multiplicative constant in the regret, which we also absorb into the factors $\Xi_1$ and $\Xi_2$ \citep[see, e.g.,][]{cutkosky2018black}.
Importantly, these types of algorithms are not Lipschitz-adaptive: they require prior knowledge of a gradient bound $G_t \ge |g_t|$ at each time step $t \ge 1$, and failing to do so results in an exponential penalty in the regret bound \citep{cutkosky2017online,mhammedi2020lipschitz}. In our unbounded-gradient setting, this limitation can be overcome via gradient clipping (in Alg.~\ref{alg:param_free_unbounded_gradients}), provided that the clipping level is carefully calibrated.

\paragraph{First result: high probability comparator-adaptive regret with known noise level.}

We first establish a guarantee for Algorithm~\ref{alg:param_free_unbounded_gradients} under the assumption that, at each time $t$, the algorithm is provided with the optimal clipping threshold $\Delta_t$, which depends on the noise parameters $(\nu,\mu)$ appearing in Assumption~\ref{assump:subexp}. This oracle assumption is subsequently removed in the following Section~\ref{section:adaptive_clipping}, where we design an adaptive gradient clipping strategy.

\begin{theorem}[High-probability $\ell_1$-comparator-adaptive regret]
\label{th:param-free_high_prob}
Let $G,C>0$, $\delta\in(0,1)$, and $N\in\mathbb N^*$.
Assume $(\ell_t)_{t \ge 1}$ satisfy Assumption~\ref{assump:exp-concave} and at each time step $t\ge 1$ we are given a stochastic gradient 
$\nabla \ell_t(\c_t) \in \mathbb R^N$ satisfying Assumption~\ref{assump:subexp} with parameter $\nu,\mu > 0$, and
\[
\|\E_{t-1}[\nabla \ell_t(\c_t)]\|_\infty \le G
\quad \text{and} \quad
\max_{1 \leq n \leq N} \E_{t-1}\!\left[|\nabla \ell_t(c_t)_n-\E_{t-1}[\nabla \ell_t(c_t)_n]|^2\right]
\le \sigma^2.
\]
Then Algorithm~\ref{alg:param_free_unbounded_gradients}, run over $T$ steps with $C_n = C$, $G_{n,t} = G$ for every $n \in [N], t\ge 1$ and clipping
margins 
\[\Delta_t = \bigg(\nu\sqrt{2\log\Big(\tfrac{(\sqrt{\pi/2}\nu+2\mu)t}{\log(\delta^{-1})}\Big)} \vee 2\mu\log\Big(\tfrac{(\sqrt{\pi/2}\nu+2\mu)t}{\log(\delta^{-1})}\Big)\bigg), \; t\ge 1
\]
satisfies, with probability at least $1-2\delta$, for all comparators
$\c \in\mathbb R^N$ such that $\|\c\|_\infty \le C$,
\[
R_T(\c)
\le O \bigg[
\|\c\|_1(G+\sigma)\sqrt{T}
+
\Big(\big(CN(G + (\nu \vee \mu)\log T)\big)^2 + \frac{1}{\alpha}\Big)\log(\delta^{-1})\bigg].
\]
% \[
% R_T(\c)
% \le O \bigg[
% \|\c\|_1\Big(
% \Xi_1(G+\sigma)\sqrt{T}
% +
% \Xi_2(G + (\nu \vee \mu)\log T)\Big)
% +
% \Big((CN(G + (\nu \vee \mu)\log T))^2 + \frac{1}{\alpha}\Big)\log(\delta^{-1})\bigg].
% \]
\end{theorem}

The proof is postponed to Appendix~\ref{appendix:proof_theo_param-free_high_prob}, where we state and prove a more general version of Theorem~\ref{th:param-free_high_prob}. This full result, stated as Theorem~\ref{th:param-free_high_prob_general}, will be instrumental for our regression setting in the next section, where we consider a multiscale approach with heterogeneous diameters. 

Note that Algorithm~\ref{alg:param_free_unbounded_gradients} requires the noise parameters $\mu,\nu>0$ to be specified in advance. In Section~\ref{section:adaptive_clipping}, we remove this requirement and develop an adaptive clipping procedure that achieves the same guarantees. In particular, the variance parameter $\sigma>0$ is one dimensional and not passed to the algorithms while it appears optimally \citep{jun2019parameter} in our regret bound $\tilde O(\|\c\|_1(G+\sigma)\sqrt{T})$ that holds in high probability.
Note also that for simplicity we stated Theorem~\ref{th:param-free_high_prob} with uniform noise in $(\nu,\mu,\sigma)$ over the coordinates. Defining a coordinate dependent clipping margin with $(\mu_n, \nu_n), n \in [N]$ would lead to an adaptive regret bound scaling as $O\big(\sum_n |c_n|(G+\sigma_n)\sqrt{T} + C^2(\nu_n \vee \mu_n)^2\log(T)^2\big)$, with high probability.

Our Theorem~\ref{th:param-free_high_prob} also parallels the results of \citet{zhang2022parameter}. As discussed in their work, even in the constrained setting where $\|\c_t\|_\infty \le C, t\ge 1$, obtaining a high-probability regret bound of order $\tilde O(\|\c\|_1\sqrt{T})$ is nontrivial. Indeed, a direct application of standard sub-exponential martingale concentration inequalities typically introduces an additional deviation term of order $\tilde O(NC\sqrt{T})$, which is too coarse to preserve the desired dependence on $\|\c\|_1$. In contrast, under Assumption~\ref{assump:exp-concave}, we are able to leverage second-order concentration tools based on self-normalized martingale inequalities. This approach allows us to control the deviation term at the cost of only an additional $O(\log^2 T)$ factor, thereby maintaining the target high-probability regret bound $\tilde O(\|\c\|_1\sqrt{T})$ without resorting to regularized or surrogate loss functions. We also note that Theorem~5 of \citet{zhang2022parameter} does not yield regret bounds with $\ell_1$-norm comparator dependence. Their high-probability guarantees are instead restricted to $\|\c\|_p$-dependence for some $p \in (1,2]$. In contrast, maintaining an optimal dependence of order $\tilde O(\|\c\|_1\sqrt{T})$ with high probability is essential for our analysis. This $\ell_1$-comparator adaptivity plays a crucial role in Section~\ref{section:risk}, where it enables the derivation of adaptive nonparametric risk bounds.

% \paragraph{Complexity of Algorithm~\ref{alg:param_free_unbounded_gradients}} Coordinate wise = clipping are very simple. If most gradients are null (as in our regression setting later), we gain over 'global' methods optimizing over the full gradient, see \cite{liautaud2025minimaxlip} Eq. (6) for a discussion.

\subsection{Adaptive clipping strategy}

\label{section:adaptive_clipping}

In Theorem~\ref{th:param-free_high_prob}, we showed that Algorithm~\ref{alg:param_free_unbounded_gradients}, when run with an optimal clipping threshold $\bar G^* = G + \Delta^*$ that exceeds the true gradient bound $G \ge \|\E_{t-1}[\nabla \ell_t(\c_t)]\|_\infty$ by a logarithmic margin $\Delta^* \propto (\nu \vee \mu)\log T$, achieves a high-probability regret guarantee of $O\big(\|\c\|_1(G+\sigma)\sqrt{T} + (CN)^2(G+\Delta^*)^2\big)$ for any $\c \in \R^N, \|\c\|_\infty \le C$. In this section, we design a procedure that achieves similar comparator-adaptive guarantees as in Theorem~\ref{th:param-free_high_prob}, \emph{without} requiring prior knowledge of the clipping margin $\Delta^*>0$, which depends on the unknown noise parameters $(\nu,\mu)$. As a consequence, the proposed strategy adapts sequentially to the level of the stochastic noise.

\paragraph{Adaptive clipping via expert aggregation.} Given $T \ge 1$, we define the following set of clipping margins to pass to Algorithm~\ref{alg:param_free_unbounded_gradients}
\begin{equation}
\label{eq:grid_margin_clipping}
\calD := \big\{\Delta = 0, \dots, \lfloor \sqrt T \rfloor +1\big\}.
\end{equation}
The procedure consists in running $|\calD| = O(\sqrt{T})$ independent instances of Algorithm~\ref{alg:param_free_unbounded_gradients}, with a bound on the true gradient for every time $t \ge 1$, $G \ge \|\E_{t-1}[\nabla \ell_t(\c_t)]\|_\infty, t\ge 1$ and with a clipping margin set at each $\Delta \in \calD$. Each instance then produces a prediction $\c_t^{\Delta}$ at time $t$ associated to each margin $\Delta \in \calD$ and we leverage an expert algorithm satisfying Assumption~\ref{assump:second_order_algo} to sequentially aggregates the predictions $(\c_t^{\Delta})_{\Delta \in \calD}$ as, for every $t\ge 1$
\begin{equation}
\label{eq:aggreg_clipping}
\hat \c_t = \sum_{\Delta \in \calD} w_t^{\Delta}\c_t^{\Delta}, \qquad \text{where} \qquad w_t^{\Delta} \ge 0\,, \; \sum_{\Delta \in \calD} w_t^{\Delta} = 1,\; t \ge 1.
\end{equation}
Remark that given the bound of Theorem~\ref{th:param-free_high_prob} scaling as $O\big(\|\c\|_1(G+\sigma)\sqrt{T} + (CN(G+\Delta))^2\big)$, a clipping margin at $\Delta = \sqrt T$ leads to a bound in $O(T)$ which is of the same order of the worst regret case bound $O(NCGT)$\footnote{For every $t\ge 1$, one has $\E_{t-1}[\ell_t(\c_t)] - \E_{t-1}[\ell_t(\c)] \leq \|\E_{t-1}[\nabla\ell_t(\c_t)]\|_\infty \|\c_t - \c\|_1 \leq 2NCG$, by convexity of $c \mapsto \E_{t-1}[\ell_t(c)]$ induced by Assumption~\ref{assump:exp-concave}, with $\alpha=0$.}, which justifies that the range $\sqrt{T}$ of the grid $\calD$ is sufficient.

\paragraph{Assumption on the expert aggregation algorithm.} We introduce the following assumption for the expert aggregation subroutine we use in \eqref{eq:aggreg_clipping}.
\begin{assumption}[Second-order adaptive algorithm]
\label{assump:second_order_algo}
Let $T, E \ge 1$ and let $\hat \nabla_1, \dots, \hat \nabla_T \in \mathbb{R}^E$ be any sequence of gradients. An expert aggregation algorithm is said to be \emph{second-order adaptive} if, starting from the uniform prior $\w_1 = (\tfrac 1 E)_{1\le e \le E}$, it produces iterates $\w_t = (w_t^{e})_{1\le e \le E}$ such that $\sum_{e\in E} w_t^{e}= 1$ and for every $e \in \{1, \dots, E\}$, the following regret bound holds:
\begin{equation*}
% \label{eq:second_order_bound}
\sum_{t=1}^T \hat \nabla_t^\top \w_t - \hat \nabla_{e,t}
\le
\Xi_3 \sqrt{\log(E) \sum_{t=1}^T (\hat \nabla_t^\top \w_t - \hat \nabla_{e,t})^2}
+
\Xi_4 \log(E) \max_{1 \le t \le T} \|\hat \nabla_t\|_\infty,
\end{equation*}
where $\Xi_3$ and $\Xi_4$ are universal and independent of the gradient range and can hide $\log\log T$ terms.
\end{assumption}
Assuming prior knowledge of a bound on the gradients, \texttt{MLProd} \citep{gaillard2014second} was the first algorithm to achieve this type of second-order regret guarantee. Building on this work, a substantial line of research has since focused on designing algorithms that enjoy such guarantees while being \emph{scale-free} and \emph{Lipschitz-adaptive} \citep[see, e.g.,][for a review]{mhammedi2019lipschitz}. Well-known instances of this class include \texttt{BOA} \citep{wintenberger2017optimal} and \texttt{Squint} \citep{koolen2015second}. Unlike standard methods that require an \emph{a priori} bound on the sup-norm of the gradients (e.g., those satisfying Assumption~\ref{assump:parameter-free-algo}), algorithms satisfying Assumption~\ref{assump:second_order_algo} adapt both to the observed gradient range $\max_{t \le T} \|\hat \nabla_t\|_\infty$ and to the cumulative quadratic variation of the prediction errors.

\paragraph{Second result: high probability comparator-adaptive regret with adaptive clipping.}

We now establish a high-probability comparator-adaptive regret bound in Theorem~\ref{theo:adapt_clipping_param_free}, for the strategy discussed in \eqref{eq:aggreg_clipping}, when the noise level is unknown. By introducing an adaptive clipping scheme, the algorithm automatically calibrates the clipping thresholds from the observed data, while retaining the same order of regret as in the oracle setting up to logarithmic factors.

\begin{theorem}
\label{theo:adapt_clipping_param_free}
Let $T, N \ge 1$ and $\delta \in (0,1)$. Assume $(\ell_t)_{t \ge 1}$ satisfy Assumption~\ref{assump:exp-concave} for some $\alpha >0$ and for every $t \ge 1, \c \in \R^N, \|\c\|_\infty \leq C, \nabla \ell_t(\c)$ satisfies Assumption~\ref{assump:subexp}, 
\[
\|\E_{t-1}[\nabla \ell_t(\c)]\|_\infty \leq G \quad \text{and} \quad \max_{1\le n \le N} \E_{t-1}[|\nabla \ell_t(\c)_n - \E_{t-1}[\nabla \ell_t(\c)_n]|^2] \le \sigma^2\,.
\] 
With sequential predictions $\hat \c_t = \sum_{\Delta \in \calD} w_t^{\Delta} \c_t^{\Delta}, t\ge 1$ as in \eqref{eq:aggreg_clipping} with $(w_t^{\Delta})_{\Delta \in \calD}$ satisfying Assumption~\ref{assump:second_order_algo}, the regret is bounded with probability $1-6\delta$ as
\begin{multline*}
    R_T(\c) \leq  O\bigg[\|\c\|_1(G+\sigma)\sqrt{T} 
    % + C(\Xi_3,\Xi_4,\alpha)NC\max(\nu,\mu)\log^2(T) \\
    % \\ + \bigg(\frac{\Xi_3^2}{\alpha} + \Xi_4NC\Big(G + 2\mu\log(TN)\Big)\bigg) \log(\sqrt T) \\
 + \bigg((\alpha + 1)  \Big(NC \Big(G + (\nu
 \vee \mu)\log(T)\Big)\Big)^2  + \frac{1}{\alpha}\bigg) \log(\delta^{-1}) \bigg]\,.
\end{multline*}
\end{theorem} 
We provide a detailed version of Theorem~\ref{theo:adapt_clipping_param_free} and a proof in Appendix~\ref{appendix:proof_theo_param-free_high_prob_adapt_clipping}.
To give some intuition, an adaptive expert algorithm of type Assumption~\ref{assump:second_order_algo} is key to derive this result: under Assumption~\ref{assump:subexp}, the maximum magnitude expert gradient $\sup_t \|\hat \nabla_t\|_\infty$ is a random variable that scales logarithmically with $T$ in high probability, and the variance (second order) term is managed with Assumption~\ref{assump:exp-concave}, enabling the recovery of high-probability regret bounds in Theorem~\ref{theo:adapt_clipping_param_free}, of type $\tilde O(\|\c\|_1(G+\sigma)\sqrt{T})$ %(up to $\log(T)^2$ terms) 
without assuming bounded gradient nor the knowledge of the noise parameters $\sigma,\nu,\mu > 0$. Interestingly, in the setting of sub-exponential noise and losses satisfying the stochastic directional derivative assumption, our result answers a challenging question raised by \citet{zhang2022parameter} as to whether it is possible to achieve a regret bound of order $\tilde O(\|\c\|_1 (G+\sigma)\sqrt{T})$ while remaining adaptive to the noise amplitude. Importantly, we believe that our adaptive clipping procedure with unknown noise level  can be applied to a broader range of learning problems. In particular, this result will play a central role in the learning of nonparametric functions with noisy observations in Section~\ref{section:risk}, and enables the derivation of minimax adaptive nonparametric risk guarantees with high probability.

\section{High-Probability Regression over Besov Spaces via Online-to-Batch}
\label{section:risk}

We now turn to the main statistical contribution of the paper, where we leverage the online learning guarantees established in the previous sections to derive high-probability minimax risk bounds for nonparametric regression over Besov spaces.

\paragraph{Setting: statistical learning.}
We observe a sample of $T \ge 1$ i.i.d.\ pairs $(X_1,Y_1),\dots,(X_T,Y_T)$ drawn from an unknown distribution on $\mathcal X \times \mathbb R$, where $\mathcal X \subset \mathbb R^d$ is compact. The data follow the nonparametric regression model
\[
Y_t = f(X_t) + \varepsilon_t, \quad t = 1,\dots,T,
\]
where $f : \mathcal X \to \mathbb R$ is an unknown regression function and $(\varepsilon_t)_{t\ge1}$ is a sequence of zero-mean noise variables with unknown variance $\sigma^2 > 0$. Our goal is to construct an estimator $\bar f_T$ that minimizes, in the minimax sense, the excess prediction risk
\[
\E_{XY}\!\big[(\bar f_T(X) - Y)^2\big]
-
\E_{XY}\!\big[(f(X) - Y)^2\big],
\]
uniformly over $f$ belonging to a Besov space (defined below). The estimator $\bar f_T$ is obtained via an online-to-batch conversion and takes the form
\[
% \textstyle
\bar f_T \;=\; \frac{1}{T}\sum_{t=1}^T \hat f_t,
\]
where each $\hat f_t$ is a wavelet-based estimator updated sequentially. In particular, each predictor $\hat f_t$ is parameterized by a collection of wavelet coefficients, which are learned online using the algorithms developed in Section~\ref{section:learning_with_known_clipping} and Section~\ref{section:adaptive_clipping}. At each round $t$, the coefficients of the predictor $\hat f_t$ are updated based on the stochastic gradient $\hat g_t \;=\; \ell_t'(\hat f_t) = 2\big(\hat f_t(X_t) - Y_t\big)$,
associated with the squared loss
$\ell_t(\hat f_t) = (\hat f_t(X_t) - Y_t)^2$.
This construction mirrors the stochastic online convex optimization framework analyzed in Section~\ref{section:parameter-free-high-prob}, as the goal is to control, for any comparator function $f$, the stochastic regret $R_T(f) = \sum_{t=1}^T
\E_{t-1}[\ell_t(\hat f_t)]-
\E_{t-1}[\ell_t(f)]$ that appears in \eqref{eq:regret_sto}.

Before describing our adaptive online wavelet regression procedure and deriving high-probability, noise-level–adaptive minimax rates matching recent optimal results \cite{devore2025optimal}, we briefly recall the necessary background on wavelet bases and Besov space approximation.

\subsection{Besov functions and wavelets.}
In the following, we provide a minimal introduction to wavelet approximation theory and Besov functions; see Appendix~\ref{appendix:wavelet} for additional details. We refer the interested reader to \cite{cohen2003numerical,gine2021mathematical,hardle2012wavelets} for a deeper introduction and comprehensive treatments of wavelets and Besov spaces.

\paragraph{Wavelet function representation.}

We consider compactly supported functions $f : \mathcal X \to \mathbb R$,  $ \mathcal X \subset \mathbb R^d$, $d\ge 1$. To design our algorithm, we rely on a multiscale representation of $f$ based on an orthonormal wavelet basis $\{\phi_{j_0,k}, \psi_{j,k}\}$. For a chosen starting scale $j_0 \in \mathbb{N}$, $f$ admits the decomposition
\begin{equation}
\label{eq:wavelet_decomposition}
% \textstyle
f
=
\sum_{k \in \bar \Lambda_{j_0}} \alpha_{j_0,k} \, \phi_{j_0,k}
+
\sum_{j = j_0}^\infty \sum_{k \in \Lambda_j} \beta_{j,k} \, \psi_{j,k}.
\end{equation}
Here, $\alphabold_{j_0} = (\alpha_{j_0,k})_{k \in \bar \Lambda_{j_0}}$ denotes the collection of \emph{scaling coefficients}, with
$|\bar \Lambda_{j_0}| \le \lambda 2^{j_0 d}$, and for each $j \ge j_0$,
$\betabold_j = (\beta_{j,k})_{k \in \Lambda_j}$ denotes the collection of
\emph{wavelet} (or \emph{detail}) coefficients at scale $j$, with
$|\Lambda_j| \le \lambda 2^{j d}$. When working with $S$-regular wavelet bases (with $S$ sufficiently large, see Definition~\ref{def:regular_wavelet} in Appendix~\ref{appendix:wavelet}), the expansion in \eqref{eq:wavelet_decomposition} converges not only in $L^2(\X)$ but also in $L^p(\X)$ for $p \geq 1$ (or the space of uniformly continuous functions) depending on whether $f \in L^p(\X), p\geq 1$. This broader convergence behavior is a key reason for adopting such regular bases.

\paragraph{Besov functions.}
Besov spaces $\B_{pq}^s$ constitute a classical family of function spaces indexed by three parameters: a smoothness parameter $s > 0$, an integrability parameter $p \in [1, \infty]$, 
and a summability parameter $q \in [1, \infty]$. This space can be intuitively viewed as the space of functions with $s > 0$ derivatives in $L^p(\X)$, with $p \ge 1$, and parameter $q \ge 1 $ allows for additional finer control of the regularity of the underlying functions. Besov spaces interpolate between Sobolev and Hölder spaces and are designed to capture both smooth and non-smooth behaviors in functions. There exist several equivalent definitions of Besov spaces (e.g. using differences, or interpolation theory): we refer to \cite{gine2021mathematical,hardle2012wavelets,triebel2006theory} for detailed and general background on Besov spaces. 
In this work, we adopt the wavelet characterization, which is particularly well suited for the analysis of our wavelet-based Algorithm~\ref{alg:online_adaptive_wavelet}.

Let $s > 0$ and let $\{\phi_{j_0,k}, \psi_{j,k}\}$ be an orthonormal $S$-regular wavelet basis with $S > s$ (see Definition~\ref{def:regular_wavelet}). 
We introduce the set of $L^\infty$-bounded Besov functions of radius $B >0$, smoothness $s > \frac{d}{p}$ and regularity parameters $1 \le p,q \le \infty$
\begin{equation}
    \label{eq:besov_norm}
\B^s_{pq}(B) = \bigg\{f \in L^\infty(\X) : \|f\|_\B := \|\alphabold_{j_0}\|_p + \Big(\textstyle \sum_{j\ge j_0}2^{-jpq(s-\frac{d}{p}+\frac{d}{2})}\|\betabold_j\|_p^q\Big)^{\frac{1}{q}} \le B, s > \frac{d}{p}\bigg\}.
\end{equation}

\paragraph{Approximation results with wavelets.}
Approximation properties of wavelet expansions are by now classical and well understood; see, e.g., \cite{devore1993constructive,cohen2003numerical,gine2021mathematical} for reviews. In particular, for functions $f$ belonging to smoothness spaces such as Besov spaces $\B^s_{pq} $ with $ s > \tfrac{d}{p} $ (so that $f \in L^\infty(\X)$), one can construct an approximation $\hat f$ using \emph{nonlinear methods} --- see \cite{devore1998nonlinear} for an introduction on nonlinear approximation theory. For instance, the so-called best $N$-term approximant $\hat f$ in a $S$-regular wavelet basis ($S>s$, see Definition \ref{def:regular_wavelet}) achieves the bound $\|f - \hat f\|_\infty \lesssim N^{-s/d}$,
where the hidden factor depends on the wavelet basis and the norm of the target function $f$. The precise construction of such $\hat f$ and justification of this rate are provided in the proof of Theorem~\ref{theo:non_param_risk} \citep[and also appears, e.g., in Proof of Theorem~1 in][]{liautaud2025minimax}.

\subsection{Algorithm: Adaptive Online Wavelet Regression}

In the adversarial setting, \citet[][Theorem~1]{liautaud2025minimax} showed that using a single wavelet predictor at a fixed uniform resolution is suboptimal, yielding a regret of order $O(\sqrt{T})$ in the high-regularity regime $s \ge \tfrac{d}{2}$. This phenomenon also appears in the batch setting, where the optimal rate is $O(T^{-\nicefrac{2s}{(2s+d)}}) < O(\tfrac{1}{\sqrt{T}})$. A workaround is to rely on a multiscale aggregation strategy that averages several wavelet predictors associated with different resolutions. Intuitively, this approach aims at selecting the resolution that best balances approximation and variance terms. As shown in Theorem~\ref{theo:non_param_risk}, such a strategy with Algorithm~\ref{alg:param_free_unbounded_gradients} achieves minimax optimal rates $O(T^{-\nicefrac{2s}{(2s+d)}})$ in the batch setting, where the losses are random through the noise $\varepsilon_t$ in the data. Remarkably, we show that our procedure also adapts to the noise level and attains an optimal dependence on the noise in the risk bound of Theorem~\ref{theo:non_param_risk}, reaching the recent noise-level aware bound \citep{devore2025optimal}, for any unknown variance $\sigma$. Our procedure builds directly on ideas of \citet{liautaud2025minimax} and we summarize the resulting algorithm in Algorithm~\ref{alg:online_adaptive_wavelet}.

\paragraph{Procedure.} Let $J_0 \ge 0$ be a maximum starting scale. At each time $t \ge 1$, we predict 
\begin{equation}
    \label{eq:predictor}
%   \textstyle 
   \hat f_t(X_t) = \sum_{e \in \mathcal E} w_{e,t} \hat f_{e,t}(X_t) \qquad \text{with } w_{e,t} \ge 0\, , \; \sum_{e\in \mathcal E} w_{e,t} = 1,
\end{equation}
where each $\hat f_{e,t}$ is a truncated wavelet estimator of type \eqref{eq:wavelet_decomposition} at some maximum level $J \ge 0$, whose coefficients are learnt with Algorithm~\ref{alg:param_free_unbounded_gradients} and is associated to a triplet $e = (j_0,\alphabold_{j_0,1},\Delta) \in [J_0]\times \mathcal A_{j_0} \times \calD =: \mathcal E$: $j_0$ is the starting scale of the wavelet expansion; $\alphabold_{j_0,1}$ are the scaling coefficients at initial time $t=1$ at level $j_0$ and set on a discretization grid $\mathcal A_{j_0} \subset  [-2^{-j_0\nicefrac{d}{2}} \, \|f\|_\infty \, \|\phi\|_1 ; 2^{-j_0\nicefrac{d}{2}} \, \|f\|_\infty \, \|\phi\|_1]$ of precision $\epsilon_{j_0} = O(2^{-j_0\nicefrac{d}{2}}/\sqrt{T})$ and of constant size $|\mathcal A_{j_0}| = O(\sqrt{T})$ for all $j_0 \in [J_0]$; and $\Delta$ is some margin clipping belonging to $\calD$ defined in \eqref{eq:grid_margin_clipping}. The total number of experts is given by
\[ 
% \textstyle
|\mathcal E|
=
(J_0+1)\,|\mathcal D|
\prod_{j_0\in[J_0]} |\mathcal A_{j_0}|
=
O\Big((J_0+1)\,T^{1+\frac{J_0}{2}}\Big).
\]
%where $J_0 = O(\log T)$.
These experts are aggregated according to \eqref{eq:predictor} using a subroutine that satisfies Assumption~\ref{assump:second_order_algo}. The complete procedure is summarized in Algorithm~\ref{alg:online_adaptive_wavelet}.

\begin{algorithm}[htbp]
\caption{Adaptive Online Wavelet Regression}\label{alg:online_adaptive_wavelet}
\SetKwInOut{Input}{Input}
\SetKwInOut{Initialize}{Initialize}
\SetKwInOut{Output}{Output}
\Input{Bound on expected derivative $\E_{t-1}[\ell'_t(\hat f_t(X_t))] \le G, t\ge 1$, bound on Besov norm $B > 0$; set of wavelet experts $(\hat f_{e})_{e\in \mathcal E}$\;}
\Initialize{Diameters $(C_{j,k}) = (B2^{-jd/2})$; uniform weights $\tilde \w_1 = (\tilde w_{e,1})_{e \in \mathcal E}$; prediction functions $(\hat{f}_{e,1})_{e \in \mathcal E}$ with initial scaling coefficients set on the grids $(\mathcal A_{j_0})_{0 \le j_0 \le J_0}$ and detail coefficients at zero\;}
\For{$t = 1$ \KwTo $T$}{
    Receive $X_t$ and predict $\hat{f}_t(X_t) = \sum_{e \in \mathcal E} w_{e,t}\hat{f}_{e,t}(x_t)$ \label{alg:line_prediction}\;
    Reveal gradient $\hat \nabla_t = \nabla_{\w_t} \ell_t\big(\sum_{e \in \mathcal E} w_{e,t} \hat{f}_{e,t}(x_t)\big)$\;
    Update $\w_{t+1} \gets \texttt{weight}(\w_t, \hat \nabla_t)$ with \texttt{weight} satisfying Assumption~\ref{assump:second_order_algo} \label{alg:line_weight}\;
    \For{$e \in \mathcal E$}{
        Reveal gradient $\hat \g_{e,t} =(g_{j,k,t})$ of coefficients in $\hat f_{e,t}$ as in \eqref{eq:gradient} \;
        Update coefficients $(c_{j,k,t})$ of $\hat f_{e,t}$ using Algorithm~\ref{alg:param_free_unbounded_gradients} with input $\hat \g_{e,t}$, expected gradient bounds $(G_{j,k,t}) = G \cdot (|\varphi_{j,k}(X_t)|)$, diameters $(C_{j,k})$ and margin clipping $\Delta_e$ \label{line:optim_param_free}\;
    }
}
\Output{$(\hat f_t)_{1\le t \le T}$ %\hat{f}_{T+1} = \sum_{e\in \mathcal E} w_{e,T+1} [\hat{f}_{e,T+1}]_B$
}
\end{algorithm}

\paragraph{Boundedness of predictors: choice of the diameters $(C_{j,k})$ and grids $(\mathcal A_{j_0})$.}
A key ingredient of our analysis is that all predictors remain uniformly bounded at every time step and we explain how this boundedness is enforced through an appropriate choice of coefficient constraints. For the coarse (scaling) coefficients, one has for all $k \in \bar\Lambda_{j_0}, |\alpha_{j_0,k}| \le 
2^{-j_0\frac{d}{2}}\|\phi\|_1 \|f\|_\infty$,
which allows us to construct the corresponding finite grids $\mathcal A_{j_0}$.
For the wavelet coefficients, assuming a bound $B>0$ on the Besov norm $\|f\|_\B$, the definition of the Besov norm~\eqref{eq:besov_norm} yields, for all $j,k$, since $s > \frac{d}{p}$
\[
|\beta_{j,k}| \le B\,2^{-j(s-\frac{d}{p}+\frac{d}{2})} \le B\,2^{-j\frac{d}{2}} =: C_{j,k}.
\]
%\vspace{-0.3cm}
Constraining the wavelet coefficients to lie in $[-C_{j,k},C_{j,k}]$ at all times ensures that the associated predictors $\hat f_{e,t}$ remain uniformly bounded in sup-norm, for all experts $e$ and all $t$, as shown in Lemma~\ref{lemma:predictor_bound}. This uniform-in-time boundedness is essential for controlling stochastic fluctuations via concentration inequalities in the high-probability regret analysis (see, e.g., the analysis in Section~\ref{section:parameter-free-high-prob}). Finally, we emphasize that the choice of the diameters $(C_{j,k})$ does not require prior knowledge of the smoothness parameter $s$ nor of the integrability parameters $(p,q)$. This adaptivity is achieved at the cost of an additional $O(\log T)$ factor in the bound. % compared to the oracle bound that would be available if $s$ and $p$ were known.

\paragraph{Computation of gradients and computational complexity.}
Denoting by $\partial_c$ the partial derivative with respect to $c$, the gradients with respect to the wavelet coefficients are obtained by the chain rule
% \vspace{-0.2cm}
\begin{equation}
% \vspace{-0.3cm}
\label{eq:gradient}
\textstyle
\hat g_{j,k,t}
=
\left[
\partial_{c}
\,\ell_t\!\left(
\hat f_t(X_t) - c_{j,k,t}\varphi_{j,k}(X_t) + c\,\varphi_{j,k}(X_t)
\right)
\right]_{c=c_{j,k,t}}
=
\ell_t'(\hat f_t(X_t))\,\varphi_{j,k}(X_t),
\end{equation}
where
$\ell_t'(\hat f_t(X_t)) = 2(\hat f_t(X_t) - Y_t)$
is the derivative of the squared loss
$\ell_t(\hat Y) = (\hat Y - Y_t)^2.$
In particular, for every $t \ge 1$, $
\E_{t-1}[\ell_t'(\hat f_t(X_t))]
=2\E_{t-1}[\hat f_t(X_t) - f(X_t) - \varepsilon_t]
\le 2(\|\hat f_t\|_\infty + \|f\|_\infty)$,
where we used $\E_{t-1}[\varepsilon_t]=0$ and the Lipschitz parameter required by Algorithm~\ref{alg:param_free_unbounded_gradients} can be set to $G = 2(\|\hat f_t\|_\infty + \|f\|_\infty)$.
Although the number of wavelet coefficients is potentially large (of order $O(2^{jd})$ at level $j$), the optimization procedure is sparse. At each time $t$, the gradient in~\eqref{eq:gradient} vanishes whenever $\varphi_{j,k}(X_t)=0$. Since wavelet basis functions are compactly supported, only a small number of coefficients are active at any given location $X_t$, and the sparsity pattern—being entirely determined by $X_t$—is known deterministically. As a result, updates and clipping are performed only on active coordinates, ensuring computational efficiency; see \cite{liautaud2025minimax} for detailed discussion on the complexity.

\paragraph{Third result: optimal noise level-aware risk bound in high probability.}

We now state our main nonparametric result, establishing a high-probability nearly minimax-optimal risk bound that is adaptive to the unknown noise level.

\begin{theorem}
\label{theo:non_param_risk}
Let $T \ge 1, B > 0, \delta \in (0,1)$, and assume $s - \tfrac{d}{p} \ge \kappa > 0$ and $\sigma \ge \sigma_0 > 0$. Let $(\hat f_e)_{e\in \mathcal E}$ be a collection of $S$-regular wavelet predictors with $S > s$ and maximum scale $J = \tfrac{S}{(2S+d)\kappa}\log(TB^2\sigma_0^{-2})$.
After $T$ rounds, Algorithm~\ref{alg:online_adaptive_wavelet} run with bounds $B$ and $G = 2(\hat B_\infty + B_\infty)$ as defined in Lemma~\ref{lemma:predictor_bound}, produces a sequence of predictors $(\hat f_t)_{1\le t \le T}$. Setting $\bar f_T = \tfrac{1}{T}\sumT \hat f_t$ we have that for any $f \in \B^s_{pq}(B)$, the following guarantee holds with probability $1-4\delta$
% \vspace{-0.1cm}
\[
% \vspace{-0.2cm}
   \|\bar f_T - f\|_2^2 \leq O\bigg(\|f\|_{\mathcal{B}}^{\frac{2d}{2s+d}} (\sigma^2)^{\frac{2s}{2s+d}} T^{-\frac{2s}{2s + d}} + B^2\frac{\log^2 T}{T}\log(\delta^{-1}) \bigg) \, .
   %B \log T \log(|\mathcal E|) + B \log^2 T \log(\delta^{-1})\bigg).
% \vspace{-0.3cm}
\]
\end{theorem}

We postpone the proof of Theorem~\ref{theo:non_param_risk} to Appendix~\ref{appendix:proof_non_param}. Remarkably, beyond its adaptivity to the unknown noise level, our algorithm achieves risk guarantees that match the minimax-optimal noise-level–aware rates recently derived in the batch setting by \citet{devore2025optimal}, who provide an explicit algorithm under knowledge of $(s,p,q,\sigma)$. Remark also that the main term scales with the true Besov norm of the function $\|f\|_\B$. Finally, the bound matches the optimal rate $\tilde O(T^{-\nicefrac{2s}{(2s+d)}})$, up to an additional $O(T^{-1}\log^2 T)$ term that reflects the cost of adaptivity induced by the aggregation procedure.
%An interesting parallel can be drawn with classical wavelet shrinkage methods: while in the batch setting, the optimal resolution level, of order $\frac{1}{2s+d}\log_2 T$, is selected implicitly through thresholding rules, our approach relies on expert aggregation to adaptively select the appropriate scale in an online fashion.
% \paragraph{Comparison with rates in online nonparametric regressin and failure of usual online-to-batch.}
Interestingly, the rates obtained in Theorem~\ref{theo:non_param_risk} can be compared with known bounds in online nonparametric regression. For the squared loss, \citet{rakhlin2014online} established minimax-optimal regret rates in the adversarial setting. When $s \ge \tfrac{d}{2}$, the average regret satisfies $\tfrac{R_T}{T} = O(T^{-\nicefrac{2s}{(2s+d)}})$, matching the i.i.d.\ minimax risk rate. However, when $s < \tfrac{d}{2}$, the regret deteriorates to $\tfrac{R_T}{T} = O(T^{-\nicefrac{s}{d}})$, which is strictly slower than the optimal batch rate $O(T^{-\nicefrac{2s}{(2s+d)}})$; see also \citet{rakhlin2017empirical} for a discussion of the gap between minimax risk and minimax regret.
This distinction highlights a key difficulty in online-to-batch conversion: even when resorting to an optimal algorithm for the online setting, such as that of \citet{liautaud2025minimax}, deriving optimal risk bounds (also in expectation) via a classical online-to-batch argument is nontrivial in the low-regularity regime $s < \tfrac{d}{2}$, where the minimax i.i.d. rate remains $O(T^{-\nicefrac{2s}{2s+d}})$. Achieving optimal high-probability risk guarantees therefore requires a careful analysis that goes beyond standard online-to-batch techniques.

\section{Conclusion}
We proposed a sequential wavelet-based algorithm that achieves minimax-optimal integrated squared error bounds with high probability, while adapting to the unknown noise level through a refined online-to-batch conversion. Our analysis also yields new high-probability, comparator-adaptive regret guarantees with $\ell_1$-norm dependence.
Several directions remain open. In particular, it would be of interest to extend our results to $L^{p'}$ risks for $p' \ge 1$ and to investigate whether the proposed approach recovers optimal noiseless recovery rates when $\sigma = 0$, as characterized in the recent work of \cite{devore2025optimal}. 

% In this work we show that our sequential wavelet-based algorithm achieves, in high probabity, minimax optimal guarentees for the integrated squared error with noise adaptivity via a refined online-to-batch conversion. We also derived new  regret bound scaling with the $\ell_1$ norm of the competitor analyzing guarentees of online algorithms in a batch regression setting with square loss. 
% Future work: extend to $L^{p'}, p' \ge 1$ and see if our algorithm matches the known hybrid rates. Also, it would be interesting to see if noiseless observation $\sigma=0$ are provided if our algorithm recovers optimal recovery minimax rates analyzed in \cite{devore2025optimal}. We left these two interesting directions for future work.

% Acknowledgments---Will not appear in anonymized version
% \acks{We thank a bunch of people and funding agency.}

% \newpage
\bibliographystyle{plainnat}
\bibliography{biblio}

\appendix

% \crefalias{section}{appendix} % uncomment if you are using cleveref

\newpage 
\input{appendix.tex}

\end{document}

%% file: mypreambule.tex
\usepackage{tikz}
\usepackage{pgfplots}
\usepackage{pgfplotstable}
\pgfplotsset{compat=1.7}
\usepgfplotslibrary{groupplots}

\usepackage[utf8]{inputenc} % allow utf-8 input
\usepackage[T1]{fontenc}    % use 8-bit T1 fonts
\usepackage{hyperref}       % hyperlinks
\usepackage{url}            % simple URL typesetting
\hypersetup{colorlinks,
            linkcolor=blue,
            citecolor=blue,
            urlcolor=magenta,
            linktocpage,
            plainpages=false}
\usepackage{booktabs}       % professional-quality tables
\usepackage{amsfonts}       % blackboard math symbols
\usepackage{amsmath} 
\usepackage{pifont} 
\usepackage{dsfont}
\usepackage{mathrsfs}
\usepackage{nicefrac}       % compact symbols for 1/2, etc.
\usepackage{microtype}      % microtypography
\usepackage{xcolor}         % colors
\usepackage{amssymb}

\usepackage{comment}

\usepackage{bm}
\usepackage[justification=justified]{caption}
\usepackage{subcaption} %several images
\usepackage{graphicx}
\usepackage{tikz}
\usepackage{float}
\usepackage{enumitem}
\usepackage{wrapfig}
\usepackage{array}

\usepackage{multirow}

\newcommand{\X}{\mathcal{X}} %instance space
\newcommand{\ind}[1]{\mathds{1}_{ #1 }}

\newcommand{\sumT}{\sum_{t=1}^T} %sum in t
\newcommand{\F}{\mathcal{F}} %function space
\newcommand{\B}{\mathcal{B}}
\newcommand{\bG}{\mathbf{G}}
\newcommand{\calD}{\mathcal{D}}
\newcommand{\w}{\mathbf{w}}
\newcommand{\g}{\mathbf{g}}
\renewcommand{\c}{\mathbf{c}}

\DeclareMathOperator{\clip}{clip}
\newcommand{\argmin}{\operatornamewithlimits{arg \, min}}

\newcommand{\R}{\mathbb{R}}

\newcommand{\sumN}{\sum_{n=1}^N}

\newcommand{\alphabold}{\boldsymbol{\alpha}}
\newcommand{\betabold}{\boldsymbol{\beta}}

\newcommand{\bxi}{\boldsymbol{\xi}}

\renewcommand{\le}{\leqslant}
\renewcommand{\leq}{\leqslant}

\renewcommand{\ge}{\geqslant}
\renewcommand{\geq}{\geqslant}

\newcommand{\E}{\mathbb{E}} %tree

\usepackage{tikz}
\usetikzlibrary{trees}
\usepackage{forest}

%% file: appendix.tex
\begin{center} 
\LARGE \textsc{Appendix}
\end{center}

\tableofcontents

\newpage 
\section{Proof of Theorem~\ref{th:param-free_high_prob}}
\label{appendix:proof_theo_param-free_high_prob}

\begin{theorem}[High-probability parameter-free regret]
\label{th:param-free_high_prob_general}
Let $N\in\mathbb N^*, C_1,\dots,C_N>0$ and $\delta\in(0,1)$.
Assume we are given noisy gradients
$\hat \g_1,\dots,\hat \g_T \in \mathbb R^N$ such that, for all $t\ge1$,
the gradient noise satisfies Assumption~\ref{assump:subexp} with parameters $\nu,\mu > 0$, and
\[
|\E_{t-1}[\hat g_{n,t}]| \le G_{n,t}, \, n \in [N]
\quad \text{and} \quad
\max_{1 \leq n \leq N} \E_{t-1}\!\left[|\hat g_{n,t}-\E_{t-1}[\hat g_{n,t}]|^2\right]
\le \sigma^2.
\]
Then Algorithm~\ref{alg:param_free_unbounded_gradients}, run with diameters $(C_n)$, gradients bounds $(G_{n,t})$ and clipping margin
\[
\Delta_t = \max\bigg\{\nu\sqrt{2\log\Big(\tfrac{(\sqrt{0.5\pi}+2\mu)t}{\log(\delta^{-1})}\Big)},2\mu\log\Big(\tfrac{(\sqrt{0.5\pi}+2\mu)t}{\log(\delta^{-1})}\Big)\bigg\}
\]
satisfies, with probability at least $1-2\delta$, for all comparators
$\c \in\mathbb R^N$ such that $|c_n|\leq C_n, n \in [N]$,
\begin{multline*}
R_T(\c)
\;\le\; O \bigg[\sum_{n=1}^N
|c_n|\Big(
\Xi_1\sqrt{\textstyle\sumT \E_{t-1}[|\hat g_{n,t}|^2]}
+
\big(\Xi_1\sqrt{\log(\delta^{-1})} + \Xi_2\big)\big(\sup_{1\leq t \leq T} G_{n,t} + \max(\nu,\mu)\log T\big)\Big) \\
+
\bigg(\frac{1}{\alpha} + \bigg(\sup_{1\leq t \leq T} \sum_{n=1}^NC_n(G_{n,t} + \max(\nu,\mu)\log T)\ind{G_{n,t}>0}\bigg)^2\bigg)\log(\delta^{-1})\bigg],
\end{multline*}
where the notation $O[\cdot]$ hides multiplicative constants and $\log \log(\delta^{-1})$ terms.
\end{theorem}
\begin{proof}
    Recall $\nabla L_t(\c_t) = \E_{t-1}[\hat \g_t]$ where we denote $\hat \g_t = (\hat g_{n,t})_{1 \leq n \leq N} = \nabla \ell_t(\c_t)$ the noisy observed gradient and $\bar \g_t = (\clip(\hat g_{n,t}, \bar G_{n,t}))_{1\le n \le N}$ the clipped version of $\hat \g_t$ at time $t$. We set $Z_t = \hat \g_t^\top(\c_t - \c)$ and $\bar Z_t = \bar \g_t^\top(\c_t - \c)$ and %\in \partial L_t(\c_t)$
   one has by stochastic $\alpha$-exp-concavity
    \begin{align}
    \sumT L_t (\c_t) -  L_t(\c) &\leq \sumT \E_{t-1}[Z_t] - \frac{\alpha}{2}\E_{t-1}[Z_t^2] \notag \\
    &\leq \underbrace{\sumT \E_{t-1}[Z_t - \bar Z_t]}_{\text{clipping bias}} + \underbrace{\sumT\E_{t-1}[\bar Z_t] - \frac{\alpha}{2}\E_{t-1}[\bar Z_t^2]}_{\text{stochastic regret}} \, \label{eq:decomposition_reg}
    \end{align}
    since $|\bar Z_t| \leq |Z_t|, t\ge 1$.

\paragraph{Bound on the stochastic regret.} For every $t\ge 1$,
\[
\sum_{n=1}^N |\bar g_{n,t}||c_{n,t} - c_n| \leq 2 \sum_{n=1}^N \bar G_{n,t} C_n \leq 2 \sup_{1 \leq t\le T} \sum_{n=1}^N \bar G_{n,t} C_n =: D.
\]
Then, by the second inequality of Lemma~\ref{lemma:concentration_second_order}, with $|\bar Z_t| \leq D, t\ge 1$ one has for every $\gamma >0$, with probability $1-2\delta$
\[
\sumT \E_{t-1}[\bar Z_t] \leq \sumT \bar Z_t + \frac{3\gamma}{2}\E_{t-1}[\bar Z_t^2] + \bigg(2D^2 + \frac{2}{\gamma}\bigg)\log(\delta^{-1}).
\]
Taking $\gamma = \frac{\alpha}{3} > 0$ cancels the second order terms and we have a stochastic regret bounded as
\begin{equation}
    \label{eq:stochastic_reg}
    \sumT \E_{t-1}[\bar Z_t] - \frac{\alpha}{2} \E_{t-1}[\bar Z_t^2] \leq \sumT \bar Z_t + \bigg(2D^2 + \frac{6}{\alpha}\bigg)\log(\delta^{-1})
\end{equation}
with probability $1-2\delta$.
The problem is now reduced to bound the linear term $\sumT \bar Z_t = \sumT \bar \g_t^\top (\c_t - \c)$ and one has 
\begin{align}
\sumT \bar \g_t^\top(\c_t - \c) &= \sum_{n=1}^N \sumT \bar g_{n,t}(c_{n,t} - c_{n}) \leq \sum_{n=1}^N |c_n|\bigg(\Xi_1\sqrt{\textstyle \sumT |\bar g_{n,t}|^2} + \Xi_2\bar G_n\bigg) \notag \\
&\leq \sum_{n=1}^N |c_n|\bigg(\Xi_1\sqrt{\textstyle \sumT 2\E_{t-1}[|\bar g_{n,t}|^2] + 2 \bar G_n^2 \log(\delta^{-1})} + \Xi_2\bar G_n\bigg) \label{eq:param_free_high_proba}
\end{align}
where first inequality is by Assumption~\ref{assump:parameter-free-algo} on iterates $(c_{n,t})$ and gradients $(\bar g_{n,t})$ with $|\bar g_{n,t}| \leq \bar G_n := \sup_{1\leq t\leq T} \bar G_{n,t}$; and last inequality holds with probability $1-\delta$, applying Lemma~\ref{lemma:ville_freedman} on $(|\bar g_{n,t}|^2 - 2\E_{t-1}[|\bar g_{n,t}|^2])/(2\bar G_n^2)$.

\paragraph{Bound on the clipping bias.} For every $t\ge1$ one has for every coordinate $n \in [N], g_{n,t} = \E_{t-1}[\hat g_{n,t}], |g_{n,t}| \leq G_{n,t}$ and
\[|\hat g_{n,t} - \bar g_{n,t}| = (|\hat g_{n,t}| - \bar G_{n,t})_+ \leq  (|\hat g_{n,t} - g_{n,t}| + |g_{n,t}| - \bar G_{n,t})_+ \leq (|\hat g_{n,t} - g_{n,t}| - (\bar G_{n,t}-G_{n,t}))_+.\]
Then, one has for every $n \in [N]$ with $\xi_{n,t} = \hat g_{n,t} - g_{n,t}, \Delta_{n,t} = \bar G_{n,t} - G_{n,t}$ and using Assumption~\ref{assump:subexp}, we have when $G_{n,t} > 0$ (else gradients are null and we take $\bar G_{n,t} = G_{n,t}$):
\begin{align} 
\E_{t-1}[|\xi_{n,t}|] &\leq \int_{\Delta_{n,t}}^\infty \mathbb P_{t-1}(|\xi_{n,t}| \ge u) \, \mathrm d u \notag \\
& \leq \int_{\Delta_{n,t}}^\infty \exp\left(-\frac{1}{2}\min\left(\frac{u^2}{\nu^2},\frac{u}{\mu}\right)\right) \, \mathrm{d}u \notag \\
&= \int_{\Delta_{n,t}}^{\frac{\nu^2}{\mu}}
\exp\!\left(-\frac{u^2}{2\nu^2}\right)\,\mathrm du
+ \int_{\Delta_{n,t}}^{\infty}
\exp\!\left(-\frac{u}{2\mu}\right)\,\mathrm du %\qquad (\bar G_{n,t} - G_{n,t} > \tfrac{\nu^2}{\mu}) 
\notag \\
&\leq \sqrt{\frac{\pi}{2}}\nu \exp\left(-\frac{\Delta_{n,t}^2}{2\nu^2}\right) + 2\mu \exp\left(-\frac{\Delta_{n,t}}{2\mu}\right) \notag \\
&\leq C(\nu,\mu) \exp\left(-\frac{1}{2}\min\left(\frac{\Delta_{n,t}^2}{\nu^2}, \frac{\Delta_{n,t}}{\mu}\right)\right),
\label{eq:bias_clipping_before_sum}
\end{align}
where $C(\nu,\mu) = \sqrt{0.5\pi}\nu + 2\mu$. Thus, the clipping bias term is bounded, with $|c_{n,t} - c_n| \leq 2C_n, n \in [N]$, as
\begin{align}
\sumT \E_{t-1}[\hat \g_t - \bar \g_t]^\top (\c_t - \c) &\leq \sumT \sum_{n=1}^N 2C_n \E_{t-1}[|\hat g_{n,t} - \bar g_{n,t}|] \notag \\ &\underset{\eqref{eq:bias_clipping_before_sum}}{\leq} 2 \sum_{n=1}^N \sumT C_n C(\nu,\mu) \exp\left(-\frac{1}{2}\min\left(\frac{\Delta_{n,t}^2}{\nu^2}, \frac{\Delta_{n,t}}{\mu}\right)\right).
\label{eq:bias_clipping}
\end{align}

\paragraph{Conclusion.} Using a union bound, we finally get with probability at least $1-3\delta$, using \eqref{eq:decomposition_reg}, \eqref{eq:stochastic_reg}, \eqref{eq:param_free_high_proba}, \eqref{eq:bias_clipping}
\begin{multline}
    R_T(\c) \leq \sum_{n=1}^N |c_n|\bigg(\Xi_1 \bigg(\sqrt{2\textstyle\sumT \E_{t-1}[|\hat g_{n,t}|^2} + \sqrt{2 \bar G_n^2 \log(\delta^{-1})} \bigg) + \Xi_2 \sup_{1\le t \le T}(G_{n,t} + \Delta_{n,t})\bigg) \\ + 2 \sumN \sumT C_n C(\nu,\mu) \exp\left(-\frac{1}{2}\min\left(\frac{\Delta_{n,t}^2}{\nu^2}, \frac{\Delta_{n,t}}{\mu}\right)\right)
    \\ +
8 \bigg(\sup_{1\le t \le T} \sum_{n=1}^N C_n(G_{n,t} + \Delta_{n,t}\ind{G_{n,t}>0})\bigg)^2 \log(\delta^{-1}) + \frac{6}{\alpha}\log(\delta^{-1}).
   \label{eq:param_free_clip_almost_done}
\end{multline}
By setting the clipping margin $\Delta_{n,t}  = \max(\nu\sqrt{2L_t}, 2\mu L_t)$ with time-dependent $L_t = \log\left(\frac{t \, C(\nu,\mu)}{\log(\delta^{-1})}\right)$, the exponential term simplifies as follows:
\begin{align*}
\sum_{t=1}^T C(\nu,\mu) \exp\left(-\frac{1}{2}\min\left(\frac{\Delta_{n,t}^2}{\nu^2}, \frac{\Delta_{n,t}}{\mu}\right)\right) &= \sum_{t=1}^T C(\nu,\mu) e^{-L_t} \\
&= \sum_{t=1}^T C(\nu,\mu) \frac{\log(\delta^{-1})}{t \, C(\nu,\mu)} \\
&= \log(\delta^{-1}) \sum_{t=1}^T \frac{1}{t} \\
&\leq \log(\delta^{-1})(\log(T)+1).
\end{align*}
Finally from \eqref{eq:param_free_clip_almost_done} we get
\begin{multline}
\label{eq:param_free_clip_complete}
R_T(\c)
\leq \sum_{n=1}^N |c_n| \bigg(\Xi_1\sqrt{2\textstyle\sumT \E_{t-1}[|\hat g_{n,t}|^2} + \big(\Xi_1\sqrt{2\log(\delta^{-1})} + \Xi_2\big) (G_n + \Delta)\bigg) \\ + 2\bigg[\sup_{1\le t \le T} \sumN C_n (\log(T) +1)\ind{G_{n,t}>0} + 4\bigg(\sup_{1\leq t\leq T}\sum_{n=1}^N C_n \big(G_{n,t} + \Delta\ind{G_{n,t}>0}\big)\bigg)^2 + \frac{3}{\alpha}\bigg]\log(\delta^{-1})
\end{multline}
where $\Delta := \sup_{1 \le t \le T} \Delta_{n,t} = \max\Big\{ \nu \sqrt{2\log\big(\tfrac{T C(\nu,\mu)}{\log(\delta^{-1})}\big)}, \; 2\mu \log\big(\tfrac{T C(\nu,\mu)}{\log(\delta^{-1})}\big) \Big\} \le O(\max(\nu,\mu)\log(T))$ and this completes the proof of Theorem~\ref{th:param-free_high_prob_general}. Theorem~\ref{th:param-free_high_prob} is obtained by setting $C_1=\dots=C_N=C$ and $G_{1,t} = \dots = G_{N,t} = G, t \ge 1$ and because for every $t \ge 1$, one has under the square root, with $g_{n,t} = \E_{t-1}[\hat g_{n,t}]$ and $|g_{n,t}| \leq G$

\begin{align*}
%\E_{t-1}[|\bar g_{n,t}|^2] \leq
\E_{t-1}[|\hat g_{n,t}|^2] &\leq \E_{t-1}[(|\hat g_{n,t} - g_{n,t}| + |g_{n,t}|)^2] \\
&\leq \E_{t-1}[|\hat g_{n,t} - g_{n,t}|^2] + \E_{t-1}[|g_{n,t}|^2] \\
&\leq \sigma^2 + G^2.
\end{align*}
\end{proof}

\newpage 
\section{Proof of Theorem~\ref{theo:adapt_clipping_param_free}}
\label{appendix:proof_theo_param-free_high_prob_adapt_clipping}
We restate and prove below a complete version of Theorem~\ref{theo:adapt_clipping_param_free}. We consider the general case with a coordinate- and time-dependent clipping rule where we have different expected gradient bounds $\bG_t := (G_{n,t})_{1\le n \le N}$ at each time $t\ge 1$. 
% Consequently, we write at each time $t\ge 1$ the clipping levels associated with an priori gradient bounds vector $\bG_t \in \R^N$ as
% \[
% \bar \bG_t = \bG_t + \Delta, \qquad \text{ with } \Delta \in \calD := \big\{0,\dots, \lfloor \sqrt T \rfloor +1 \big\}\,, \quad t \ge 1.
% \]
% The final output is, at each time $t \ge 1$
% \begin{equation}
%     \label{eq:aggreg_clipping_multi}
%     \hat \c_t = \sum_{\Delta \in \calD} w_t^{\Delta}\c_t^{\Delta}
% \end{equation}
% and we denote $\c_t^{\Delta}$ the output at time $t$ of Algorithm~\ref{alg:param_free_unbounded_gradients} with clipping threshold $\bar \bG_t = \bG_t + \Delta$. 

\begin{theorem}
\label{theo:adapt_clipping_param_free_general}Let $T, N \ge 1, (C_n) >0$ and $\delta \in (0,1)$. Assume $(\ell_t)_{t \ge 1}$ satisfy Assumption~\ref{assump:exp-concave} and for every $t \ge 1, \c \in \R^N, \|\c\|_\infty \leq C, \nabla \ell_t(\c)$ satisfies Assumption~\ref{assump:subexp} and
\[
|\E_{t-1}[\nabla \ell_t(\c)_n]| \leq G_{n,t}, \, n \in [N] \quad \text{and} \quad \max_{1\le n \le N} \E_{t-1}[|\nabla \ell_t(\c)_n - \E_{t-1}[\nabla \ell_t(\c)_n]|^2] \le \sigma^2\,.
\] 
With sequential predictions $\hat \c_t = \sum_{\Delta \in \calD} w_t^{\Delta} \c_t^{\Delta}, t\ge 1$ as in \eqref{eq:aggreg_clipping} with $(w_t^{\Delta})_{\Delta \in \calD}$ satisfying Assumption~\ref{assump:second_order_algo}, the regret is bounded with probability $1-6\delta$ and for any $\c \in \R^N, |c_n| \le C_n$, as
\begin{multline*}
    \sumT L_t(\hat \c_t) - L_t(\c) \leq O\bigg[\sum_{n=1}^N |c_n|\bigg(\Xi_1 \sqrt{\textstyle\sumT \E_{t-1}[|\hat g_{n,t}|^2} + \Big(\Xi_1\sqrt{\log(\delta^{-1})} + \Xi_2\Big)\Big(\sup_{1\le t \le T} G_{n,t} + \mu \log T\Big)\bigg) \\ 
    + \bigg(\frac{\Xi_3^2}{\alpha} + \Xi_4\sumN C_n\Big(\sup_{1\le t \le T} G_{n,t} + 2\mu\log(TN)\ind{G_{n,t}>0}\Big)\bigg) \log(\sqrt{T}+1) \\
 + \bigg((\alpha + 1)  \sup_{1 \le t \le T} \bigg(\sumN C_n \Big(G_{n,t} + 2\mu\log(TN)\ind{G_{n,t}>0}\Big)\bigg)^2  + \frac{1}{\alpha}\bigg) \log(\delta^{-1}) \bigg]\,,
\end{multline*}
where the notation $O[\cdot]$ hides multiplicative constants.
\end{theorem} 
\begin{proof}[{\bfseries of Theorem~\ref{theo:adapt_clipping_param_free_general}}]
At each time $t \ge 1, \c_t^{\Delta} \in \R^N$ is the output of Algorithm~\ref{alg:param_free_unbounded_gradients} with clipping margins $\Delta_t = \Delta, t \ge 1$ and $\hat \c_t = \sum_{\Delta \in \calD} w_t^{\Delta} \c_t^{\Delta}$ is the average prediction over $\calD$. One can decompose the regret, for every $\Delta^* \in \calD$,  as
\begin{equation*}
    \sumT L_t\bigg(\sum_{\Delta \in \calD} w^{\Delta}_t\c^{\Delta}_t\bigg) - L_t(\c) = \underbrace{\sumT L_t(\hat \c_t) - L_t(\c_t^{\Delta^*})}_{\text{expert aggregation regret}}  + \underbrace{\sumT L_t(\c_t^{\Delta^*})-L_t(\c)}_{\text{regret of Alg.~\ref{alg:param_free_unbounded_gradients} with $\Delta^*$ \quad $:= R_T(\c)$}}.
\end{equation*}

\paragraph{Expert aggregation regret.}

Define $\w_t = (w_t^{\Delta})_{\Delta\in \calD}$ and one has
\[
\hat \nabla_t = \nabla_{\w_t}\,\ell_t\Bigg(\sum_{\Delta \in \calD} w_t^{\Delta} \c_t^{\Delta}\Bigg)
=
\bigl(\hat \g_t^\top \c_t^{\Delta}\bigr)_{\Delta \in \calD}
\]
where in this part we denote $\hat \g_t = \nabla \ell_t(\hat \c_t) \in \R^N$ the gradient of $\ell_t$ evaluated at $\hat \c_t$, which satisfies Assumption~\ref{assump:subexp} and $\E_{t-1}[\nabla \ell_t(\hat \c_t)_n]| \leq G_{n,t}$.
Therefore, denoting $\hat \nabla_t = (\hat \nabla_t^\Delta)_{\Delta \in \calD}$ and letting $Z_t = \hat \nabla_t^\top \w_t - \hat \nabla_t^{\Delta^*}$, by Assumption~\ref{assump:exp-concave} one has
    \begin{align}
        \sumT L_t\Big(\sum_{\Delta \in \calD} w_t^{\Delta} \c_t^{\Delta}\Big) - L_t(\c_t^{\Delta^*}) &\leq \sumT \E_{t-1}[Z_{t}] - \frac{\alpha}{2}\E_{t-1}[Z_{t}^2] \notag \\
        &\leq \sumT Z_{t} + \frac{\gamma}{2} Z_t^2 + \frac{\gamma - \alpha}{2} \E_{t-1}[Z_{t}^2] + \frac{2}{\gamma}\log(\delta^{-1}) \label{eq:bound_expert_reg}
\end{align}
where last inequality holds with high probability $1-\delta$ for any $\gamma >0$ by Lemma~\ref{lemma:concentration_second_order}. 
Moreover, one can bound the first order term by Assumption~\ref{assump:second_order_algo} with $Z_t = \hat \nabla_t^\top \w_t - \hat \nabla_t^{\Delta^*}$
    \[
    \sumT Z_t \leq \Xi_3\sqrt{\log(|\calD|)\sumT Z_t^2} + \Xi_4\log(|\calD|) \sup_{1\leq t \leq T} \|\hat \nabla_t\|_\infty.
    \]
Further, by Young's inequality, we have for every $\eta >0$
    \[
    \Xi_3 \sqrt{\log(|\calD|) \sumT  Z_{t}^2} \leq \frac{\Xi_3^2 \log(|\calD|)}{2\eta} + \frac{\eta}{2}\sumT  Z_{t}^2.
    \]
And in bound \eqref{eq:bound_expert_reg} we have
\begin{equation}
\label{eq:bound_first_order}
\sumT Z_t + \frac{\gamma}{2}Z_t^2 \leq \frac{\Xi_3^2 \log(|\calD|)}{2\eta} + \frac{\eta+\gamma}{2}\sumT  Z_{t}^2 + \Xi_4\log(|\calD|) \sup_{1\leq t \leq T} \|\hat \nabla_t\|_\infty.
\end{equation}
Now we want to relate $\sumT Z_t^2$ to its clipped version using the fact that $\hat \g_t$ has coordinate with sub-exponential noise (Assumption~\ref{assump:subexp}). Indeed, remark that for every $t \ge 1$
\[
Z_t =\hat \nabla_t^\top \w_t - \hat \nabla_t^{\bar G^*} = \sum_{\Delta \in \calD} \sumN \hat g_{n,t}c_{n,t}^\Delta w_t^\Delta - \sumN \hat g_{n,t}c_{n,t}^{\Delta} = \sumN \hat g_{n,t}\bigg(\sum_{\Delta \in \calD} w_t^\Delta c_{n,t}^\Delta - c_{n,t}^{\Delta^*}\bigg)
\]
is still sub-exponential as a sum of sub-exponential random variables $\hat g_{n,t}$. % {\color{orange} écrire un lemme avec les bons paramètres}. 
We define a clipped version of $Z_t, \bar Z_t = \sumN \bar g_{n,t}\bigg(\sum_{\Delta \in \calD} w_t^\Delta c_{n,t}^\Delta - c_{n,t}^{\Delta^*}\bigg)$ where $\bar g_{n,t} = \clip(\hat g_{n,t},\tau_{n,t}), n\in [N]$ for some $\tau_{n,t} > 0$ to be optimized in the proof and that only appears in the analysis. Then, we have
\begin{equation}
\label{eq:equality_clip}
    \sumT Z_t^2 = \sumT \bar Z_t^2
\end{equation}
with probability $\mathbb P(\forall 1\le t \le T, Z_t = \bar Z_t)$.
By the union bound,
\[
\mathbb P\!\left(\forall\, 1 \le t \le T,\; Z_t = \bar Z_t \right)
=
1 - \mathbb P\!\left(\exists\, 1 \le t \le T:\; Z_t \neq \bar Z_t \right)
\;\ge\;
1 - \sum_{t=1}^T \mathbb P_{t-1}\!\left(Z_t \neq \bar Z_t\right).
\]
In particular, for every $t \ge 1$,
\begin{align*}
\mathbb P_{t-1}\!\left(Z_t \neq \bar Z_t\right)
&=
\mathbb P_{t-1}\!\left(
\sum_{n=1}^N \hat g_{n,t}
\Bigl(\sum_\nu w_t^\nu c_{n,t}^\nu - c_{n,t}^{\bar G^*}\Bigr)
\neq
\sum_{n=1}^N \bar g_{n,t}
\Bigl(\sum_\nu w_t^\nu c_{n,t}^\nu - c_{n,t}^{\bar G^*}\Bigr)
\right) \\
&\le
\sum_{n=1}^N
\mathbb P_{t-1}\!\left(
\hat g_{n,t}
\Bigl(\sum_\nu w_t^\nu c_{n,t}^\nu - c_{n,t}^{\bar G^*}\Bigr)
\neq
\bar g_{n,t}
\Bigl(\sum_\nu w_t^\nu c_{n,t}^\nu - c_{n,t}^{\bar G^*}\Bigr)
\right) \\
&=
\sum_{n=1}^N
\mathbb P_{t-1}\!\left(\hat g_{n,t} \neq \bar g_{n,t}\right) \\
&=
\sum_{n=1}^N
\mathbb P_{t-1}\!\left(|\hat g_{n,t}| > \tau_{n,t} \right) \\
&\leq \sum_{n=1}^N
\mathbb P_{t-1}\!\left(|\hat g_{n,t} - \E_{t-1}[\hat g_{n,t}]|  > \tau_{n,t} - G_{n,t}\right) \\
&\leq \sumN  C(\nu,\mu)\exp\left(-\frac{1}{2}\min\left(\frac{(\tau_{n,t} - G_{n,t})^2}{\nu^2}, \frac{\tau_{n,t} - G_{n,t}}{\mu}\right)\right),
\end{align*}
where second last inequality is because $|\E_{t-1}[\hat g_{n,t}]| \leq G_{n,t}, n \in [N]$, and last inequality is by Assumption~\ref{assump:subexp} with $C(\nu,\mu) = \sqrt{0.5\pi}\nu + 2\mu$. Finally, taking
\begin{multline*}
\tau_{n,t} = G_{n,t} + \nu \sqrt{2\log\big(\tfrac{NT C(\nu,\mu)}{\delta}\big)} \vee 2\mu \log\big(\tfrac{NT C(\nu,\mu)}{\delta}\big) \\
\implies C(\nu,\mu)\exp\left(-\frac{1}{2}\min\left(\frac{(\tau_{n,t} - G_{n,t})^2}{\nu^2}, \frac{\tau_{n,t} - G_{n,t}}{\mu}\right)\right) \leq \frac{\delta}{TN}
\end{multline*}
and Equality \eqref{eq:equality_clip} holds with probability at least $1- \sumT \sumN \delta \frac{1}{TN} = 1 - \delta$.
With $\sum_{\Delta \in \calD} w_t^{\Delta} = 1, t\ge 1$ and such clipping, we have for every $t \ge 1$
\[
|\bar Z_{t}| = \bigg|\sumN \bar g_{n,t}
\Bigl(\sum_{\Delta \in \calD} w_t^{\Delta} c_{n,t}^{\Delta} - c_{n,t}^{\Delta^*}\Bigr)\bigg| \leq \sup_{1\le t \le T} \sumN 2C_n \big(G_{n,t} + \Delta(\tau_{n,t})\big) := D.
\]
where $\Delta(\tau_{n,t}) := \nu \sqrt{2\log\big(\tfrac{NT C(\nu,\mu)}{\delta}\big)} \vee 2\mu \log\big(\tfrac{NT C(\nu,\mu)}{\delta}\big) \leq 2(\mu,\nu)\log\big(\tfrac{NT C(\nu,\mu)}{\delta}\big) = O(\log(TN))$ provided that $\log\big(\tfrac{NT C(\nu,\mu)}{\delta}\big) \ge \frac{1}{2}$, i.e. $T \ge \exp(1/2)\delta/(NC(\nu,\mu))$.\\
By Corollary~\ref{cor:concentration_expectation}, with $0 \le \bar Z_t^2 \le D^2$, one has with probability $1-\delta$
\[
\sumT \bar Z_t^2 \leq 2 \sumT \E_{t-1}[\bar Z_t^2] + 2D^2\log(\delta^{-1}).
\]
Finally, with probability $1-2\delta$ (union bound), one has in Equation~\eqref{eq:bound_first_order}
\[
\frac{\eta + \gamma}{2} \sumT Z_t^2 \leq (\eta + \gamma) \sumT \E_{t-1}[\bar Z_t^2] + (\eta + \gamma)D^2\log(\delta^{-1}),
\]
which gives after combining Equations~\eqref{eq:bound_expert_reg}, \eqref{eq:bound_first_order} and last inequality, since $\frac{\gamma-\alpha}{2}\E_{t-1}[Z_t^2]\leq \frac{\gamma-\alpha}{2}\E_{t-1}[\bar Z_t^2]$,$\gamma < \alpha$ and $|\bar Z_t| \leq |Z_t|$ 
\begin{multline} 
\label{eq:bound_reg_expert_almost_done}
    \sumT L_t(\hat \c_t) - L_t(\c_t^{\Delta^*}) \leq \frac{2\eta + 3\gamma - \alpha}{2} \sumT \E_{t-1}[\bar Z_t^2]  + \frac{\Xi_3^2 \log(|\calD|)}{2\eta} \\ + \Xi_4\log(|\calD|) \sup_{1\leq t \leq T} \|\hat \nabla_t\|_\infty + \bigg((\eta + \gamma)D^2 + \frac{2}{\gamma}\bigg)\log(\delta^{-1})
\end{multline}
with probability $1-3\delta$.
Finally, one has to control $\sup_{1 \le t \le T} \|\hat \nabla_t\|_\infty$ with high probability. Indeed,
\[
\|\hat \nabla_t\|_\infty =  \sup_{\Delta \in \calD} |\hat \g_t^\top \c_t^\Delta| \leq \sup_{1 \le t \le T} \sumN C_n |\hat g_{n,t}|
\]
since $c_{n,t}^{\calD} \in [-C_n,C_n]$ for every $\Delta \in \calD$ and taking the sup over $1\le t \le T$, one has 
\begin{equation}
    \label{eq:bound_proba_sup_nabla}
\sup_{1\le t \le T} \|\hat \nabla_t\|_\infty \le \sup_{1\le t \le T} \sumN C_n|\hat g_{n,t}| = \sup_{1\le t \le T} \sumN C_n|\bar g_{n,t}|
\end{equation}
where last inequality holds probability at least $1-\delta$ with again $\bar g_{n,t} = \clip(\hat g_{n,t},\tau_{n,t})$ using the same clipping threshold as before since \[
\mathbb P\bigg(\forall \, 1\le t \le T,\;\sumN C_n|\hat g_{n,t}| = \sumN C_n|\bar g_{n,t}|\bigg) \ge 1 - \sumT \sumN \mathbb P(|\hat g_{n,t}| > \tau_{n,t}) \ge 1 - \delta.
\]
Then, taking $\eta = \tfrac{\alpha - 3\gamma}{2}, \gamma < \frac{\alpha}{3} < \alpha$ cancels the second order terms in \eqref{eq:bound_reg_expert_almost_done} and gives, combined with \eqref{eq:bound_proba_sup_nabla}
\begin{multline} 
    \label{eq:bound_reg_expert_final}
    \sumT L_t(\hat \c_t) - L_t(\c_t^{\Delta^*}) \leq \frac{\Xi_3^2 \log(|\calD|)}{\alpha - 3\gamma} + \Xi_4\log(|\calD|) \sumN C_n\tau_n\\ + \bigg(\frac{\alpha-\gamma}{2}  \bigg(2\sumN C_n \tau_n \bigg)^2 + \frac{2}{\gamma}\bigg)\log(\delta^{-1})
\end{multline}
with probability $1-4\delta$ (union bound) and $\tau_n = \sup_{1 \le t \le T} G_{n,t} + O(\max(\nu,\mu)\log\big(\frac{TN}{\delta}\big))$. %and $D=2\sumN C_n \big(\sup_{1\le t \le T} G_{n,t} + 2\mu\log\big(\tfrac{NT}{\delta}\big)\big)$.

\paragraph{Regret $R_T(\c)$ of Algorithm~\ref{alg:param_free_unbounded_gradients} with clipping margin $\Delta^*$.}

Let \[
\Delta(\nu,\mu) =  \max\bigg\{\nu\sqrt{2\log\Big(\tfrac{(\sqrt{0.5\pi}+2\mu)T}{\log(\delta^{-1})}\Big)},2\mu\log\Big(\tfrac{(\sqrt{0.5\pi}+2\mu)T}{\log(\delta^{-1})}\Big)\bigg\}
\]
be the optimal clipping margin in Theorem~\ref{theo:adapt_clipping_param_free} and define $\Delta^* = \argmin_{\Delta > \Delta(\nu,\mu) \in \calD} |\Delta - \Delta(\nu,\mu)|$ the closest upper margin in the grid $\calD$. \\
From Equation~\eqref{eq:param_free_clip_almost_done} in Proof of Theorem~\ref{th:param-free_high_prob_general} (Appendix~\ref{appendix:proof_theo_param-free_high_prob}), with clipping margin $\Delta^*$, $G_{n} = \sup_{1 \le t \le T} G_{n,t}$ and $\hat g_{n,t} = \nabla \ell_t(\c_t^{\Delta^*}), |\E_{t-1}[\hat g_{n,t}]|\leq G_{n,t}, n \in [N], t \ge 1$ (and assumptions of Theorem~\ref{th:param-free_high_prob_general} are all satisfied), one has

\begin{multline}
\label{eq:bound_param_free_adapt_clipping}
    R_T(\c^\mu) \leq \sum_{n=1}^N |c_n|\bigg(\Xi_1\sqrt{2\textstyle\sumT \E_{t-1}[|\hat g_{n,t}|^2} + \Big(\Xi_1\sqrt{2 \log(\delta^{-1})} + \Xi_2\Big) (G_n + \Delta^*) \bigg) \\ 
    + 2 \sumN \sumT C_n C(\nu,\mu) \exp\left(-\frac{1}{2}\min\left(\frac{(\Delta^*)^2}{\nu^2}, \frac{\Delta^*}{\mu}\right)\right)
    \\ +
8 \bigg(\sup_{1\le t \le T} \sum_{n=1}^N C_n(G_{n,t} + \Delta^*\ind{G_{n,t}>0})\bigg)^2 \log(\delta^{-1}) + \frac{6}{\alpha}\log(\delta^{-1})
\end{multline}
with probability $1-2\delta$.
Since $\Delta^* > \Delta(\nu,\mu)$ we still have (see below \eqref{eq:param_free_clip_almost_done}) \[
\sumT C(\nu,\mu) \exp\left(-\frac{1}{2}\min\left(\frac{(\Delta^*)^2}{\nu^2}, \frac{\Delta^*}{\mu}\right)\right) \le  \log(\delta^{-1})(\log(T)+1)
\]
and by the precision of the grid $\calD$, one has $|\Delta^* - \Delta(\mu,\nu)| \leq 1, \Delta(\mu,\nu) = O((\nu \vee \mu)\log(T))$ and then the deviation term 
\[
O\bigg(\sup_{1\le t \le T} \sum_{n=1}^N C_n(G_{n,t} + \Delta^*\ind{G_{n,t}>0})\bigg) = O\bigg(\sup_{1\le t \le T} \sum_{n=1}^N C_n(G_{n,t} + \Delta(\nu,\mu))\ind{G_{n,t}>0}\bigg)
\]
is still of the same order $O\Big(\log^2 T \log(\delta^{-1})\big(\sumN C_n\big)^2\Big)$.

\paragraph{Conclusion.} Combining \eqref{eq:bound_reg_expert_final} and \eqref{eq:bound_param_free_adapt_clipping} one has with probability $1 - 6\delta$ and $\gamma = \tfrac \alpha 4 < \tfrac \alpha 3$
\begin{multline*}
    \sumT L_t(\hat \c_t) - L_t(\c^\mu) \leq \sum_{n=1}^N |c_n|\bigg[\Xi_1 \sqrt{2\textstyle\sumT \E_{t-1}[|\hat g_{n,t}|^2}
    + \bigg(\Xi_1\sqrt{2 \log(\delta^{-1})} + \Xi_2\bigg)(G_{n} + \tau)\bigg] \\ + \frac{4\Xi_3^2 \log(|\calD|)}{\alpha} + \Xi_4\log(|\calD|) \sumN C_n\bigg(G_n + \tau \bigg) \\
 + \bigg[\frac{3\alpha}{8}  \bigg(2\sumN C_n \bigg(G_n + \tau \bigg)^2 + 4 (\log(T)+1)\sumN C_n\\+ 8\bigg(\sum_{n=1}^N C_n (G_{n} + \tau)\bigg)^2 + \frac{14}{\alpha}\bigg] \log(\delta^{-1}).
\end{multline*}
where $\tau = O\big(\max(\nu,\mu)\log\big(\frac{TN}{\delta}\big)\big)$ and this concludes the proof with $|\mathcal D| \le \sqrt{T} +1$.
\end{proof}

\paragraph{Remark.} Considering sparse expected gradients $\g_t = \E_{t-1}[\hat \g_t]$ the sum $\sumN$ in the deviation terms can be reduced, at each time as a sum over $n \in \operatorname{supp}(\g_t)$, provided that $G_{n,t} = 0$ when $g_{n,t}=0$. This will be useful in particular in the non-parametric setting where $N$ is exponential, but only a small fraction $|\operatorname{supp}(\g_t)|$ of coordinates are non zero (active) at time each time $t$, hence a much smaller diameter in the deviation terms.

\newpage 
\section{Proof of Theorem~\ref{theo:non_param_risk}}
\label{appendix:proof_non_param}

\begin{proof}[{\bfseries of Theorem~\ref{theo:non_param_risk}}] Let $f \in \B^s_{pq}(B)$ be fixed and $Y = f(X) + \varepsilon$ with $\E[\varepsilon|X]=0$ and $\E[\varepsilon^2|X] = \sigma^2 < \infty$. Recall that the training data set $((X_t,Y_t))_{1 \le t \le T} \sim (X,Y)$ is i.i.d. and with same law as the test pair $(X,Y)$.

\paragraph{Decomposition of risk.}

For the square loss, one has for any predictor $\hat f$ this nice decomposition of the expected risk
\begin{align}
    \E_{XY}[(\hat f(X) - Y)^2]   &= \E_X[\E[(\hat f(X) - Y)^2|X]] \notag \\
                            &= \E_X[\E[(\hat f(X) - f(X) - \xi)^2|X]] \notag \\
                            &= \E_X[\E[(\hat f(X) - f(X))^2|X]] + \E_X[\E[\xi^2|X]] - 2 \E_X[\E[(\hat f(X)-f(X))\xi|X]] \notag \\
                            &= \E_X[(\hat f(X) - f(X))^2] + \E[\xi^2] \notag \\
                            &= \|\hat f - f\|^2_2 + \sigma^2. \label{eq:decomposition_L2_risk}
\end{align}
Going back to the excess risk and letting $\hat f = \bar f_T = \tfrac{1}{T}\sumT \hat f_t$ in \eqref{eq:decomposition_L2_risk}, after direct algebraic calculations one has 
\begin{align}
    \|\bar f_T - f\|_2^2 &= \E_{XY}[(\bar f_T(X) - Y)^2] - \E_{XY}[(f(X) - Y)^2] \notag \\
    &\leq \frac{1}{T} \sumT \E_{XY}[(\hat f_t(X) - Y)^2] - \E_{XY}[(f(X) - Y)^2], \label{eq:online-to-batch_Jensen}
\end{align}
by Jensen's inequality and convexity of the square loss.
Further, since $(X_t,Y_t) \sim P_{(X,Y)}$ and is independent of $\F_{t-1}$ and $\hat f_t$ is $\F_{t-1}$ measurable, one has for every $t \ge 1$
\[\E_{XY}[(\hat f_t(X) - Y)^2] = \E_{t-1}[(\hat f_t(X_t) - Y_t)^2],\]
where we used the notation $\E_{t-1}[\cdot] = \E[\cdot \mid \F_{t-1}]$. And finally,
\begin{equation}
    \label{eq:online_to_batch_final}
    \|\bar f_T - f\|_2^2 \leq \frac{1}{T} \underbrace{\sumT \E_{t-1}[(\hat f_t(X_t) - Y_t)^2] - \E_{t-1}[(f(X_t) - Y_t)^2]}_{=: R_T(f)}
\end{equation}
The problem now turns out to analyse the (stochastic) regret $R_T(f)$ and this situation mirrors the problem studied in Section~\ref{section:parameter-free-high-prob} with losses $L_t(\hat f_t) = \E_{t-1}[\ell_t(\hat f_t(X_t))], t\ge 1$ where $\ell_t(\hat f_t(X_t)) = (\hat f_t(X_t) - Y_t)^2$ and $\ell_t$'s randomness is due to $Y_t$'s one.\\
The quantity of interest $R_T(f)$ then decomposes for every oracle estimator $\hat f^*$, into an estimation regret term $R_T(\hat f^*)$ plus a bias (or approximation regret) term
\[
R_T(f) = R_T(\hat f^*) + \sumT \E_{t-1}[\ell_t(\hat f^*(X_t))] - \E_{t-1}[\ell_t(f(X_t))] \underset{\eqref{eq:decomposition_L2_risk}}{=} R_T(\hat f^*) + T \|\hat f^* - f\|_2^2,
\]
with $\ell_t(\hat f) = (\hat f - Y_t)^2$.

\paragraph{Step 1: Bound on the estimation regret $R_T(\hat f^*)$.} 

One has $\hat f_t(X_t) = \sum_e w_{e,t} \hat f_{e,t}(X_t)$ (average over all experts). Consider $e^*=(j_0^*,\Delta^*,\alphabold^*) \in [J_0] \times \calD \times \mathcal A_{j_0^*}$ ($j_0^*$ oracle starting scale, $\alphabold^*$ best scaling coefficients on the grid $\mathcal A_{j_0^*}$, and $\Delta^*$ the best clipping margin). We then define the expert with best parameters $e^*$ at every time $t \ge 1, \hat f_{e^*,t}$ and the oracle regressor $\hat f_{e^*}^*(X_t)$ with oracle coefficients $(c_{j,k})$ and oracle parameters $e^*$. We have at every time $t \ge 1$, the following decomposition
\begin{align*} 
R_T(\hat f^*) &= \sumT L_t(\hat f_t(X_t)) - L_t(\hat f^*) \\ 
&= \underbrace{\sumT L_t(\hat f_t) - L_t(\hat f_{e^*,t}(X_t))}_{\text{expert aggregation regret} \quad =:R_1} +  \underbrace{ \sumT L_t(\hat f_{e^*,t}(X_t)) - L_t(\hat f^*_{e^*}(X_t))}_{\text{regret of Alg.\ref{alg:param_free_unbounded_gradients}} \quad =: R_2}.
% &= \underbrace{\sum_{e^* \in \mathcal E^*}\sumT (L_t(\hat f_t) - L_t(\hat f_{e^*,t}))\ind{X_t \in \X_{e^*}}}_{\text{expert aggregation regret} \quad =:R_1} +  \underbrace{\sum_{e^* \in \mathcal E^*} \sumT (L_t(\hat f_{e^*,t}) - L_t(\hat f^*))\ind{X_t \in \X_{e^*}}}_{\text{regret of Alg.\ref{alg:param_free_unbounded_gradients}} \quad =: R_2}.
\end{align*}
% with general $L_t(\hat f) = \E_{t-1}[\ell_t(\hat f(X_t))]$ such that
% \[
% L_t(\hat y_1) - L_t(\hat y_2) \leq \nabla_{\hat y_1}L_t(\hat y_1)^\top(\hat y_1 - \hat y_2) - \frac{\alpha}{2}\E_{t-1}[(\nabla \ell_t(\hat y_1)^\top(\hat y_1 - \hat y_2))^2].
% \]
% In particular this enables to consider any regression losses $\ell_t(\hat y_t) = \ell(\hat y_t,Y_t) = |\hat y_t - Y_t|^p, p\ge 2$ that are (deterministic) exp-concave functions in $\hat y_t$.

\begin{enumerate} 
\item \textbf{Bound on the expert aggregation regret $R_1$.}
    Let $\hat g_t = \ell_t'(\hat f_t(X_t)) = 2(\hat f_t(X_t) - Y_t)$ be the noisy observed gradient of $L_t'(\hat f_t(X_t))$. In particular, $\hat g_t$ satisfies Assumption~\ref{assump:subexp} and 
    \[
    |\E_{t-1}[\hat g_t]| = 2\E_{t-1}[\hat f_t(X_t) - f(X_t) - \varepsilon_t]| \leq 2(\hat B_\infty + B) =: G.
    \] %The analysis goes similarly to that of Theorem~\ref{theo:adapt_clipping_param_free}.
    The gradient experts $\hat \nabla_t = (\nabla_{e,t})_{e \in \mathcal E}$ writes $\hat \nabla_{e,t} = \hat g_t \cdot \hat f_{e,t}(X_t)$ and $|\E_{t-1}[\hat \nabla_{e,t}]| \leq G \hat B_{\infty}$ where $\hat B_\infty$ is the sup-norm on the experts, bounded as in Lemma~\ref{lemma:predictor_bound}.
Therefore, letting $Z_{e^*,t} = \hat \nabla_t^\top \w_t - \hat \nabla_{e^*,t}$, by Assumption~\ref{assump:exp-concave} on $(\ell_t)$ one has
    \begin{align}
        \sumT L_t(\hat f_t) - L_t(\hat f_{e^*,t})
        &\leq \sumT \E_{t-1}[Z_{e^*,t}] - \frac{\alpha}{2}\E_{t-1}[Z_{e^*,t}^2] \notag \\
        &\leq \sumT Z_{e^*,t} + \frac{\gamma}{2} Z_{e^*,t}^2 + \frac{\gamma - \alpha}{2} \E_{t-1}[Z_{e^*,t}^2] + \frac{2}{\gamma}\log(\delta^{-1}) \label{eq:bound_expert_R1}
\end{align}
where last inequality holds with high probability $1-\delta$ for any $\gamma >0$ by Lemma~\ref{lemma:concentration_second_order}.
%and since at every time $t \ge 1$, there exists a unique $e^*_t \in \mathcal E^*$ such that $X_t \in \X_{e^*_t}$ and $\sumT\sum_{e^* \in \mathcal E^*} Z_{e^*,t}\ind{X_t \in \X_{e^*}} = \sumT Z_{e^*_t,t}$. 
Moreover, %conditionally on each $\{X_t \in \X_{e^*}\}, e^* \in \mathcal E^*$,
one can bound the first order term by Assumption~\ref{assump:second_order_algo} with $Z_{e^*,t} = \hat \nabla_t^\top \w_t - \hat \nabla_{e^*,t}$
    \[
    \sumT Z_{e^*,t} \leq \Xi_3\sqrt{\log(|\mathcal E|)\sumT Z_{e^*,t}^2} + \Xi_4\log(|\mathcal E|) \sup_{1\leq t \leq T} \|\hat \nabla_t\|_\infty.
    \]
Further, by Young's inequality, we have for every $\eta >0$
    \[
    \Xi_3 \sqrt{\log(|\mathcal E|) \sumT  Z_{e^*,t}^2} \leq \frac{\Xi_3^2 \log(|\mathcal E|)}{2\eta} + \frac{\eta}{2}\sumT  Z_{e^*,t}^2.
    \]
And in bound \eqref{eq:bound_expert_R1} we have %, conditionally on $\{X_t \in \X_{e^*}\}$,
\begin{multline}
\label{eq:bound_first_order_R1}
\sumT Z_{e^*,t} + \frac{\gamma}{2}Z_{e^*,t}^2 \leq \frac{\Xi_3^2 \log(|\mathcal E|)}{2\eta} + \Xi_4\log(|\mathcal E|) \sup_{1\leq t \leq T} \|\hat \nabla_t\|_\infty \\ +\frac{\eta+\gamma}{2} \sum_{e^* \in \mathcal E^*} \sumT  Z_{e^*,t}^2.
\end{multline}
Further, since $\hat g_t$ satisfies Assumption~\ref{assump:subexp} with some parameters $(\nu,\mu)$, we have $\sup_{1 \le t \le T} \|\hat \nabla_t\|_\infty = \hat B_\infty \sup_{1 \le t \le T} |\hat g_t|$ that can be bounded, with probability $1-\delta$, as 
\begin{equation}
\label{eq:sup_nabla}
\sup_{1 \le t \le T} \|\hat \nabla_t\|_\infty \le \tau \hat B_\infty
\end{equation}
because $\forall 1 \le t \le T, g_t = \clip(g_t,\tau)$ with probability $1-\delta$ as with $\tau = G + O((\nu \vee \mu)\log(T/\delta))$ under Assumption~\ref{assump:subexp}. Then, with $Z_{e^*,t} = \nabla_t^\top \w_t - \hat \nabla_{e^*,t} = \hat g_t (\hat f_t(X_t) - \hat f_{e^*,t}(X_t))$, one has also with high probability $1-\delta$,
\[
\sumT Z_{e^*,t} = \sumT \clip(g_t,\tau) (\hat f_t(X_t) - \hat f_{e^*,t}(X_t)),
\]
Setting $\bar Z_{e^*,t} := \clip(g_t,\tau) (\hat f_t(X_t) - \hat f_{e^*,t}(X_t))$, we can use Corollary~\ref{cor:concentration_expectation}, with $0 \le (\bar Z_{e^*,t})^2 \le 4\tau^2 \hat B^2_\infty$ and we finally have with probability $1-2\delta$ (union bound)
\begin{equation}
\label{eq:concentration_order_2}
\sumT Z_{e^*,t}^2 \le 2\sumT \E_{t-1}[\bar Z_{e^*,t}^2] + 8\tau^2\hat B^2_\infty\log(\delta^{-1}).
\end{equation}
Plugging \eqref{eq:bound_first_order_R1}, \eqref{eq:sup_nabla} and \eqref{eq:concentration_order_2} in \eqref{eq:bound_expert_R1}, we get with probability $1-4\delta$ (union bound)
\begin{multline}
    \label{eq:bound_expert_R1_almost_done}
    R_1 \leq \frac{2\eta + 3\gamma - \alpha}{2} \sumT  \E_{t-1}[\bar Z_{e^*,t}^2] \\
    + \log(|\mathcal E|)\bigg(\frac{\Xi_3^2 }{2\eta} + \Xi_4 \tau \hat B_\infty \bigg) + 2\bigg(4\tau^2\hat B_\infty^2 + \frac{1}{\gamma}\bigg)\log(\delta^{-1}) 
\end{multline}
taking $0 < \gamma < \alpha$ so that $\tfrac{\gamma-\alpha}{2}\E_{t-1}[Z_{e^*,t}^2] \le \tfrac{\gamma-\alpha}{2}\E_{t-1}[\bar Z_{e^*,t}^2]$. Finally, taking $0 < \gamma < \tfrac{\alpha}{3}, \eta =  \frac{\alpha - 3\gamma}{2} > 0$ cancels the second order term and we have, with probability $1-4\delta$
\begin{equation}
    \label{eq:bound_expert_R1_final}
    R_1 \leq \log(|\mathcal E|)\bigg(\frac{\Xi_3^2 }{\alpha - 3\gamma} + \Xi_4 \tau \hat B_\infty \bigg) + 2\bigg(4\tau^2\hat B_\infty^2 + \frac{1}{\gamma}\bigg)\log(\delta^{-1}).
\end{equation}
with $\tau = G + O((\nu \vee \mu)\log(T/\delta))$

\item \textbf{Bound on $R_2$ via Theorem~\ref{th:param-free_high_prob}.} 
% We have
% \begin{equation*}
% R_{2}
% :=
% \sum_{e^* \in \mathcal E^*} \sum_{t=1}^T
% \Big(
% L_t(\hat f_{e^*,t}(X_t))
% -
% L_t(\hat f^*_{e^*}(X_t))
% \Big)
% \ind{X_t \in \X_{e^*}}
% =
% \sum_{e^* \in \mathcal E^*} \sum_{t\in T_{e^*}}
% L_t(\hat f_{e^*,t}(X_t))
% -
% L_t(\hat f^*_{e^*}(X_t))
% \end{equation*}
% where $T_{e^*} = \{1 \leq t \leq T : X_t \in \X_{e^*}\}$, and where $\hat f^*_{e^*}$ is the local wavelet predictor with oracle coefficients $c_{j,k}$. 
We can apply Theorem~\ref{th:param-free_high_prob} to bound $R_2$.
Recall that $\hat f^*_{e^*}$ is an oracle wavelet regressor (oracle coefficients, and oracle starting resolution) and
\[
\hat f_{e^*,t}
=
\underbrace{\sum_{k \in \bar \Lambda_{j_0^*}} \alpha_{j_0^*,k,t}\,\phi_k}_{\text{scaling part}} + \underbrace{\sum_{j=j_0^*}^{j_0^*+J} \sum_{k \in \Lambda_{j}} \beta_{j,k,t}\,\psi_{j,k}}_{\text{detail part}}.
\]
where $\alpha_{j_0^*,k,t}$ denotes scaling coefficients at time $t$, starting at $\alpha_{j,k,1}^* \in \mathcal A$, and $\beta_{j,k}, j\ge j_0^*$ are wavelet coefficients up to some optimized truncation level $J=\frac{S}{(2S+d)\kappa}\log(TB^2\sigma_0^{-2})$ (large enough to kill the approximation term).
Applying Theorem~\ref{th:param-free_high_prob} on coefficients $\{\alpha_{j,k},\beta_{j,k}\}$, with oracle clipping margin $\Delta_{e^*} \propto \max(\nu,\mu)\log(T) \in \calD$, with probability $1-2\delta$
\begin{multline*}
\sumT L_t(\hat f_{e^*,t}(X_t)) - L_t(\hat f^*_{e^*}(X_t)) \leq C \sum_{j,k} |c_{j,k}-c_{j,k,1}| \sqrt{\textstyle \sumT \E_{t-1}[\hat g_{j,k,t}^2]} \\ + O\Big(\Big(\sup_{1\le t \le T} \sum_{j,k} C_{j,k}(G_{j,k,t} + \max(\nu,\mu)\log(T))\ind{G_{j,k,t} > 0}\Big)^2\log(\delta^{-1})\Big)
\end{multline*}
where $c_{j,k}$ stands for either $\alpha_{j,k}$ or $\beta_{j,k}$ and $C$ is deduced from Theorem~\ref{th:param-free_high_prob_general} and can include $\log T$ terms.
We study the bound in 4 parts:
\begin{itemize}
    \item \underline{Scaling coefficients:} since the scaling coefficients at level $j_0^*$ starts on a grid of precision $\epsilon_{j_0^*} = B_\infty \|\phi\|_1(2^{j_0d/2}\sqrt T)^{-1}$ (the grid is of size $O(\sqrt T)$ for every level $j_0^*$)
\begin{align*}
\sum_{k \in \bar \Lambda_{j_0^*}} \underbrace{|\alpha_{j_0^*,k} - \alpha_{j_0^*,k,1}|}_{\le \frac{\epsilon_{j_0^*}}{2}}\sqrt{\sumT \E_{t-1}[\hat g_{j,k,t}^2]} &\leq \epsilon_{j_0^*} |\bar \Lambda_{j_0^*}|^{\frac{1}{2}}\sqrt{\sum_{k \in \bar \Lambda_{j_0^*}} \sumT \E_{t-1}[\hat g_{j,k,t}^2]} \\
&\leq \frac{B_\infty \|\phi\|_1 \lambda^{\frac{1}{2}}}{\sqrt T} \sqrt{M_\phi \|\phi\|_\infty T} = (\lambda M_\phi \|\phi\|_\infty)^{\frac{1}{2}}
\end{align*}
where we used Jensen's inequality, $|\bar \Lambda_{j_0^*}| \le \lambda 2^{j_0^*d}$, the form of the gradients $\hat g_{j,k,t} = \hat g_t \cdot \phi_k(X_t)$ and the fact that $\sum_k |\phi_k(X_t)|^2 \le \|\phi\|_\infty \cdot \|\sum_k |\phi(\cdot - k)|\|_\infty = \|\phi\|_\infty \cdot M_\phi < \infty$.
    \item \underline{Detail coefficients:} We use the same oracle as in \cite[Theorem~1]{liautaud2025minimax} that was designed for a control on the $L^\infty$ bias and works also for the $L^2$-norm. We have a linear part and non linear part in the wavelet oracle: we keep all the detail coefficients in between $[j_0^*,j_0^* + J^*]$ up to some level $j_0^*+J^*$ and then retain only greatest coefficients $\beta_{j,k}$ from level $[j_0^*+J^*, J]$.
    \begin{itemize}
        \item \underline{Linear part: $j = j_0^*,\dots,J^*+j_0^*$} For the linear part, denoting $\betabold_j = (\beta_{j,k})_{k \in \Lambda_j}, j=j_0^*,\dots j_0^*+J$, one has
\begin{align*}
\sum_{j = j_0^*}^{j_0^* + J^*} &\sum_{k \in \Lambda_{j}} |\beta_{j,k}|\sqrt{\textstyle \sumT \E_{t-1}[|\hat g_{j,k,t}|^2]} \\
&\leq \sum_{j = j_0^*}^{j_0^* + J^*}\|\boldsymbol{\beta}_j\|_p  |\Lambda_j|^{(\frac{1}{2}-\frac{1}{p})_+} \textstyle \sqrt{\sumT \sum_{k \in \Lambda_j(e^*)} \E_{t-1}[|\hat g_{j,k,t}|^2]} \\
&\le \sum_{j = j_0^*}^{j_0^* + J^*} \|\boldsymbol{\beta}_j\|_p |\Lambda_j|^{(\frac{1}{2}-\frac{1}{p})_+}  \textstyle \sqrt{\sumT \E_{t-1}[|\hat g_t|^2] \|\sum_{k \in \Lambda_j} |\psi_{j,k}(\cdot)|^2\|_\infty} \\
&\leq \sum_{j = j_0^*}^{j_0^* + J^*} \|\boldsymbol{\beta}_j\|_p |\Lambda_j|^{(\frac{1}{2}-\frac{1}{p})_+}  2^{\frac{jd}{2}} \textstyle \sqrt{M_\psi \|\psi\|_\infty \sumT \E_{t-1}[|\hat g_t|^2] } \\
&\leq \lambda^{\frac{1}{2}}\|f\|_{\B} {\textstyle \sqrt{M_\psi \|\psi\|_\infty \sumT \E_{t-1}[|\hat g_t|^2]}}
\sum_{j = j_0^*}^{j_0^* + J^*} 2^{-j\beta}
\end{align*}
where first inequality is by Hölder's inequality (with $p\ge 1$) and convexity on $\sum_k$, and last inequality applies Hölder's inequality on $\sum_j$ (with $q \ge 1$), $|\Lambda_j|\leq \lambda 2^{jd}$ and $M_\psi = \|\sum_k |\psi(\cdot-k)|\|_\infty < \infty$ and $\|f\|_\B$ is the Besov norm of the function $f$ defined in \eqref{eq:besov_norm} and we defined $\beta = s-\frac{d}{p} -(\frac{d}{2}-\frac{d}{p})_+$.
Further, we have for every $t\ge 1$
\[
\E_{t-1}[\hat g_t^2] = 4(\E_{t-1}[(\hat f_t(X_t) - f(X_t))^2] + \sigma^2) = 4(L_t(\hat f_t(X_t)) - L_t(f(X_t)) + \sigma^2)
\]
and plugging this in the previous bound gives 
\[
C \|f\|_{\B} \textstyle \Big(\sqrt{\sumT (L_t(\hat f_t(X_t)) - L_t(f(X_t)))} + \sigma \sqrt{T}\Big)
\sum_{j = j_0^*}^{j_0^* + J} 2^{-j\beta}.
\]
where $C$ can be deduced from the last lines.

\item \underline{Detail coefficients in the nonlinear part:} we consider only a sparse representation of the function in levels $[j_0^* + J^*,J]$, and we keep only an oracle set of coefficients $\Lambda^* = \{(j,k), J^* \le j \le J, \text{greatest coefficients } |\beta_{j,k}|\}$. Similarly to \cite[Theorem~1]{liautaud2025minimax}, retaining only $|\Lambda^*| = 2^{(j_0^*+J^*)d}$ gives a regret of the same order of that of the linear part above. Indeed, with $\Lambda^* = \cup_{j=J^*+j_0^*+1}^J \Lambda^*_j$ where $|\Lambda^*_j| \leq |\Lambda_j|$ is now the oracle sparse set made of positions $k$ at level $j$. One has on the levels $j = J^*+j_0^*+1,\dots,J$,
\begin{multline*}
    C \sum_{J^*+j_0^*<j\leq J} \|\betabold_j\|_{p}\ |\Lambda_j^*|^{(\frac{1}{2} - \frac{1}{p})_+} \sqrt{\textstyle \sum_{k\in \Lambda^*_j} \sumT \E_{t-1}[|\hat g_{j,k,t}|^2]} \\
    \leq C \textstyle\Big(\sqrt{\sumT (L_t(\hat f_t(X_t)) - L_t(f(X_t)))} + \sigma \sqrt{T}\Big) \sum_{J^*+j_0^*<j\leq J}\|\betabold_j\|_p|\Lambda_j^*|^{(\frac{1}{2}-\frac{1}{p})_+}2^{\frac{dj}{2}}
\end{multline*}
Then, using Hölder's inequality with $q \ge 1$ one has
\begin{align*}
\sum_{J^*+j_0^*<j\leq J}2^{\frac{dj}{2}}\|\betabold_j\|_p|\Lambda_j^*|^{(\frac{1}{2}-\frac{1}{p})_+} &= \sum_{J^*+j_0^*<j\leq J}2^{j(s+\frac{d}{2}-\frac{d}{p})}\|\betabold_j\|_p \cdot 2^{-j(s-\frac{d}{p})}|\Lambda_j^*|^{(\frac{1}{2}-\frac{1}{p})_+}  \\
&\leq \|f\|_{\B} \sum_{J^*+j_0^*<j\le J} 2^{-j(s-\frac{d}{p})}|\Lambda_j^*|^{(\frac{1}{2}-\frac{1}{p})_+}.
\end{align*}
Finally, since $\sum_{J^*<j\le J} |\Lambda_j^*| = |\Lambda^*|$, one has
\begin{align*}
\sum_{J^*+j_0^*<j\le J} 2^{-j(s-\frac{d}{p})}|\Lambda_j^*|^{(\frac{1}{2}-\frac{1}{p})_+} &\leq |\Lambda^*|^{(\frac{1}{2}-\frac{1}{p})_+} \sum_{J^*+j_0^*<j\le J} 2^{-j(s-\frac{d}{p})} \\
&\leq |\Lambda^*|^{(\frac{1}{2}-\frac{1}{p})_+} \frac{2^{-(J^*+j_0^*)(s-\frac{d}{p})}}{2^{s-\frac{d}{p}}-1},
\end{align*}
by Hölder's inequality in the case $p \ge 2$ and since $s > \frac{d}{p}$ and taking $|\Lambda^*| = 2^{d(J^* +j_0^*)}$ entails a nonlinear regret of 
\[
% C \|f\|_\B\textstyle\Big(\sqrt{\sumT (L_t(\hat f_t(X_t)) - L_t(f(X_t)))} + \sigma \sqrt{T}\Big) \sum_{J^*+j_0^*<j\leq J}\|\betabold_j\|_p|\Lambda_j^*|^{(\frac{1}{2}-\frac{1}{p})_+}2^{\frac{dj}{2}} \leq
C \|f\|_\B\textstyle\Big(\sqrt{\sumT (L_t(\hat f_t(X_t)) - L_t(f(X_t)))} + \sigma \sqrt{T}\Big) 2^{-(J^* + j_0^*)(s-\frac{d}{p}-(\frac{d}{2}-\frac{d}{p})_+)}
\]
which is of the same order of the regret for the linear part. 
\end{itemize}

\item \underline{Deviation term:} 
Let $\varphi$ denote either the scaling basis function $\phi$ or the wavelet $\psi$. With $G_{j,k,t} = G\, |\varphi_{j,k}(X_t)| = G 2^{\frac{jd}{2}} |\varphi(2^{j}X_t - k)|$ and $C_{j,k} = B2^{-\frac{jd}{2}}$ at any time $t \ge 1$, one has
\[
\sup_{1 \le t \le T} \sum_{j,k} |C_{j,k} G_{j,k,t}| \leq G B \sup_{x} \sum_{j,k} |\varphi(x - k)| \leq \sum_{j=j_0^*}^{j_0^*+J} \sup_x \sum_{k \in \Lambda_{j}} |\varphi(x-k)| %\leq LB \sum_{n\in \Lambda_{e^*}} \sum_{j=j_0(n)}^{J+j_0(n)} \bigg(\sum_{k \in \Lambda_{e^*,j}}} |\varphi(x - k)|\bigg)\ind{j=j(n)} 
\leq GB M_\varphi (J+1),
\]
and setting $N_\varphi = \sup_x \sum_{k} \ind{|\varphi(x - k)|>0} < \infty$ the number of wavelets involved at each scale (of order $S^d$)
\[
\sup_{1 \leq t \leq T} \sum_{n\in \Lambda_{e^*}} C_n\ind{G_{n,t}>0} \leq B  \sum_{j=j_0(e^*)}^{j_0(e^*)+J} 2^{-\frac{jd}{2}} \sup_x \sum_{k \in \Lambda_{e^*,j}}  \ind{|\varphi(x-k)|>0} \leq  \frac{B N_\varphi}{1 - 2^{-d/2}}.
\]

Finally the deviation term is of order 
\[
D := O\bigg(\Big(GBM_\varphi(J+1) + \frac{B N_\varphi}{1 - 2^{-d/2}}\max(\nu, \mu)\log(T)\Big)^2\log(\delta^{-1})\bigg),
\]
where $J$ will be of order $\log T$.
\item \underline{Conclusion on $R_2$:} One has with probability $1-2\delta$, summing all the previous bounds:
\[
R_2 \le C\|f\|_\B \textstyle \Big(\sqrt{\sumT (L_t(\hat f_t(X_t)) - L_t(f(X_t)))} + \sigma \sqrt{T}\Big)
\sum_{j = j_0^*}^{j_0^* + J} 2^{-j\beta} + D
\]
\end{itemize}
\end{enumerate}

\paragraph{Step 2: Bound on the approximation term.}
Recall $\Lambda^*$ denotes the set of indices corresponding to the $|\Lambda^*|$ largest wavelet coefficients (in absolute value) among all $(v_{j,k}=\beta_{j,k}2^{j(s+d/2-d/p)})$ with $j \in [J^*+j_0^* + 1, J]$ and $k \in \Lambda_j$. Let $j > J^*+j_0^*$. We have that
\begin{multline*}
\label{eq:bound_lp}
   \sum_{J^*+j_0^* < j \leq J}\sum_{k \in \Lambda_j} |v_{j,k}|^p \leq (J-(J^*+j_0^*))^{(1-\frac{p}{q})_+}\bigg(\sum_{J^*+j_0^* < j \leq J}\bigg(\sum_{k\in \Lambda_j}|v_{j,k}|^p\bigg)^{\frac{q}{p}}\bigg)^{\frac{p}{q}} \\
   \leq \big[(J-(J^*+j_0^*))^{(\frac{1}{p}-\frac{1}{q})_+}\|f\|_{\B}\big]^p =: C_f^p,
\end{multline*}
and
\begin{equation*}
%\label{eq:bound_beta_nonlinear}
|\Lambda^*|\cdot \min_{(j,k)\in \Lambda^*} |v_{j,k}|^p \leq \sum_{(j,k)\in \Lambda^*} |v_{j,k}|^p \leq C_f^p < \infty.
\end{equation*}
In particular since $\forall (j,k) \not \in \Lambda^*, |v_{j,k}| \leq \min_{(j',k')\in \Lambda^*} |v_{j',k'}|$ one has
\begin{equation}
\label{eq:bound_beta_nonlinear}
\forall (j,k) \not \in \Lambda^*, \; |\Lambda^*| |v_{j,k}|^p \leq C_f^p \implies \forall (j,k) \not \in \Lambda^*, \; |\beta_{j,k}|\leq C_f 2^{-j(s+d/2-d/p)}|\Lambda^*|^{-\frac{1}{p}}.
\end{equation}
And, by Parseval's inequality, with $s'= d-d/p+d/2$
\begin{align}
\|\hat f^* - f\|_2^2  &= \sum_{(j,k) \not\in \Lambda^*} |\beta_{j,k}|^2 + \sum_{j> J} |\beta_{j,k}|^2 \\
&\leq C_f^2 |\Lambda^*|^{-\frac{2}{p}} \sum_{j=J^*+j_0^*+1}^{J} \sum_k  2^{-2js'} + \sum_{j >J} \sum_k \|f\|_\B^22^{-2js'}
 \qquad \leftarrow \text{by \eqref{eq:bound_beta_nonlinear}} \notag \\ 
&\leq  C_f^2 |\Lambda^*|^{-\frac{2}{p}} \frac{2^{-2(J^*+j_0^*)(s-\frac{d}{p})}}{2^{(s-\frac{d}{p})}-1} + \|f\|_\B^2 \frac{2^{-2J(s-\frac{d}{p})}}{2^{(s-\frac{d}{p})}-1} \qquad \leftarrow \text{ replacing $s'$ and with $s > \frac{d}{p}$.} \notag 
\\
&= \frac{1}{2^{s-\frac{d}{p}}-1} \big(C_f^2 2^{-2J^*s} +  \|f\|_\B^2 2^{-2J(s-d/p)}\big) \qquad \leftarrow |\Lambda^*|=2^{j_0^*+J^*}
\label{eq:bound_non_linear_approx}
\end{align}
Finally one has
\begin{equation}
\label{eq:expected_approx}
R_3 := T\|\hat f^*_{e^*} - f\|_2^2 \leq C \|f\|_\B^2 T 2^{-2(J^*+j_0^*)s} + \frac{\|f\|^2_\B}{2^{s-d/p}} 2^{-2J(s-d/p)}
\end{equation}
% where $\hat f^* = f_{(J^*+j_0^*)} + f_{\Lambda^*}\ind{p<2}$ and 
where $C \leq (J-(J^*+j_0^*))^{2(\frac{1}{p} - \frac{1}{q})_+}/(2^{2(s-d/p)}-1)$. 
In the next section, we take $J$ high enough so that the bias terms in $2^{-2J(s-d/p)}$ vanishes.

\paragraph{Step 3: Optimization of the resolutions $J$, $J^*$ and $j_0^*$.}

Summarizing the previous steps, we have with probability at least $1-6\delta$ (via a union bound over $R_1, R_2$, and the concentration terms):
\[
R_T(f) \le R_1 + R_2 + R_3 + D
\]
where $R_1 = O((\log|\mathcal{E}| + \log^2 T\log\delta^{-1})$ is the expert regret and $D=O(\log^2 T \log\delta^{-1})$ is the deviation term in estimating the wavelet coefficients. From here, we assume that $J >0$ is high enough so that the term $2^{-2J(s-d/p)}$ in $R_3$ vanishes.\\
According to the sign of $\beta$, one has, with $M^* = j_0^* + J^*$ the maximum resolution,
\begin{itemize}
    \item Case $\beta > 0$
    \begin{itemize}
        \item  $\beta = s - d/2 > 0$ and $p > 2$, i.e. $s > d/2$. 
        One has $\sum_{j=j_0^*}^{j_0^* + J^*} 2^{-\beta j} \leq C' 2^{-j_0^*(s-d/2)}$ and \[R_T(f) \leq C'2^{-j_0^*(s-\frac{d}{2})} \|f\|_{\B} \bigg( \sqrt{R_T(f)} + \sigma \sqrt{T}\bigg) + C \|f\|_{\B}^2 T 2^{-2(j_0^*+ J^*)s} + (R_1 + D) \]
        \item $\beta = s - d/p > 0$ and $p < 2$.
         One has $\sum_{j=j_0^*}^{j_0^*+ J^*} 2^{-\beta j} + 2^{-(s-\frac{d}{p})(j_0^*+ J^*)} \leq C'2^{-j_0^*(s-d/p)}$ and \[R_T(f) \leq C' 2^{-j_0^*(s-d/p)}\|f\|_{\B} \bigg( \sqrt{R_T(f)} + \sigma \sqrt{T}\bigg) + C \|f\|_{\B}^2 T 2^{-2(j_0^*+ J^*)s} + (R_1 + D)\]
    \end{itemize}
    \item $\beta = s-d/2 = 0$ and $p \ge 2$. One has $\sum_{j=j_0^*}^{j_0^*+ J^*} 2^{-\beta j} \leq C'J^*$ and \[R_T(f) \leq C'J^*\|f\|_{\B} \bigg( \sqrt{R_T(f)} + \sigma \sqrt{T}\bigg) + C \|f\|_{\B}^2 T 2^{-2(j_0^*+ J^*)s} + (R_1 + D) \]
    \item $\beta = s-d/2 < 0$ and $p > 2$, i.e. $s < d/2$. 
        One has $\sum_{j=j_0^*}^{j_0^*+ J^*} 2^{-j\beta} \leq C' 2^{-(j_0^*+ J^*)(s-\frac{d}{2})}$ 
        \[
        R_T(f) \leq C' \|f\|_{\B} 2^{-(j_0^*+J^*)(s-\frac{d}{2})} \bigg( \sqrt{R_T(f)} + \sigma \sqrt{T}\bigg) + C \|f\|_{\B}^2 T 2^{-2(j_0^*+ J^*)s} + (R_1 + D)
        \]
\end{itemize}

Let $U = \sqrt{R_T(f)}$. In all regimes, the bound satisfies a quadratic inequality $U^2 \le AU + A\sigma\sqrt{T} + R_3 + (R_1+D)$. Applying Young's inequality ($AU \le \frac{1}{2}U^2 + \frac{1}{2}A^2$), we obtain:
\begin{equation}
R_T(f) \le A^2 + 2A\sigma\sqrt{T} + 2R_3 + 2(R_1 + D). \label{eq:master_bound}
\end{equation}

In each case, the terms $A^2 + 2A\sigma\sqrt{T} + 2R_3$ dominates and are of opposite variation. We set the maximal resolution $M^* = j_0^* + J^*$ an we have in each case:

\begin{itemize}
    \item \textbf{Case $\beta = s - d/2 < 0$:} In this regime ($s < d/2$), the variance term is dominated by the fine-scale resolution $M^*$. With $A = C'\|f\|_{\mathcal{B}} 2^{-M^*(s-\frac{d}{2})}$, the regret is of order:
    \[
    \|f\|_{\mathcal{B}}^2 2^{-2M^*(s-\frac{d}{2})} + \|f\|_{\mathcal{B}} 2^{-M^*(s-\frac{d}{2})}\sigma \sqrt{T} + \|f\|_{\mathcal{B}}^2 T 2^{-2M^*s}.
    \]
    Balancing the variance and bias by setting $\sigma\|f\|_{\mathcal{B}}\sqrt{T} 2^{-(s-\frac{d}{2})M^*} \asymp T \|f\|^2_{\mathcal{B}} 2^{-2M^*s}$ yields the optimal resolution:
    \[
    M^* = \left\lceil \frac{1}{2s+d}\log_2\big(T \|f\|^2_{\mathcal{B}}\sigma^{-2}\big)\right \rceil
     \leq J \]
    where we assume $\sigma^2 \le T$.
    Plugging this back into the bound gives:
    \[
    R_T(f) \leq O\bigg(\|f\|_{\mathcal{B}}^{\frac{2d}{2s+d}} (\sigma^2)^{\frac{2s}{2s+d}} T^{\frac{d}{2s + d}}\bigg) + 2(R_1 + D).
    \]
    Note that $A^2 = O(T^{\frac{d-2s}{2s+d}})$, which is strictly smaller than $O(T^{\frac{d}{2s+d}})$.

    \item \textbf{Case $\beta = 0$:} For $s = d/2$, the variance term $A$ scales with the number of levels, $A = C'\|f\|_{\mathcal{B}} J^*$. Setting $J^* \propto \frac{1}{d}\log_2(T) \leq J$ and choosing $j_0^* = M^* - J^*$, with $M^*$ as defined above, ensures the optimal regret up to a logarithmic factor:
    \[
    R_T(f) \leq O\bigg(\|f\|_{\mathcal{B}}^{\frac{2d}{2s+d}} (\sigma^2)^{\frac{2s}{2s+d}} T^{\frac{d}{2s + d}}\log_2(T)\bigg) + 2(R_1 + D) = O\left( \sqrt{T} \log_2(T) \right).
    \]

    \item \textbf{Case $\beta > 0$:} This regime corresponds to $s - d/2 > 0$ (or $s - d/p > 0$ for $p < 2$). With $A = C' 2^{-j_0^*(s-d/2)}\|f\|_{\mathcal{B}}$, the dominant terms are:
    \[
    \|f\|_{\mathcal{B}}^2 2^{-2j_0^*(s-d/2)} + \|f\|_{\mathcal{B}} 2^{-j_0^*(s-d/2)} \sigma \sqrt{T} + \|f\|_{\mathcal{B}}^2 T 2^{-2(j_0^*+ J^*)s}.
    \]
    Since $s > d/2$, all terms decrease with $j_0^* \in [0,J_0]$. Here, the variance term $2A\sigma\sqrt{T}$ dominates $R_3$ when $j_0^*$ is small. To achieve the minimax rate, we can choose $j_0^*=$ (sup on the grid of starting scale) to suppress the $\sqrt{T}$ noise growth, with $j_0^*=\frac{1}{s-\frac{d}{2}}
\log_2\big(
\|f\|_{\mathcal B}^{\frac{2s}{2s+d}}
\,\sigma^{\frac{d}{2s+d}}
\,T^{\frac{s-\frac{d}{2}}{2s+d}}\big) \propto \frac{1}{2s+d}\log_2(T)$
    and we get the same rate as before. 
    \item \textbf{Maximum scale $J$:} In $R_3$ also appears a bias, so we have to ensure that we take a maximum scale $J$ that is high enough so that the bias term $O(T2^{-2J(s-d/p)})$ vanishes. To control the bias term, we choose a resolution level $J$, assuming $\|f\|_\B \le B \leq \sqrt{T}$ and $\sigma \ge \sigma_0 >0$.
\[
J \;=\;  \left\lceil \frac{S}{(2S+d)\,\kappa}\,\log_2 (TB^2\sigma_0^{-2}) \right\rceil \ge \left\lceil \frac{S}{(2S+d)\,\kappa}\,\log_2 (T\|f\|^2_\B\sigma^{-2}) \right\rceil \ge 1
\]
assuming $\|f\|_\B \ge \frac{1}{\sqrt{T}}$ and $\sigma \ge \sqrt{T}$. This choice of maximum scale $J$ ensures that
\[
\|f\|^2_\B T\,2^{-2J(s-d/p)} \;\le\; T\,2^{-2J\kappa}
\;\le \; O(T^{\frac{d}{2s+d}}).
\]

\end{itemize}

\paragraph{Conclusion.} Summing with $R_1 + D$ and dividing all the preceding rates by $T$ concludes the proof.
% {\color{red} Case $\sigma = 0$, on paie seulement le biais d'approximation comme dans \cite[Eq. (3.4)]{devore2025optimal}}
\end{proof}

\newpage

\section{Technical lemmas}

We recall several standard concentration inequalities and technical lemmas that are used in the proofs.

\begin{lemma}[Exponential inequality]
\label{lemma:exp_inequality}
Let $Z$ be a positive and bounded random variable such that $0 \le Z \leq D$ for some $D >0$. Then, for every $0 \le \eta < 1 / D$ one has
\[\E\bigg[\exp\bigg(\eta Z - \frac{\eta \E[Z]}{1 - \eta D}\bigg)\bigg] \leq 1.\]
\end{lemma}
\begin{proof}
    Let $\eta >0$. Expanding the exponential, one has by linearity of the expectation
    \begin{align*}
    \E[\exp(\eta Z)] = \E\bigg[ \sum_{k\ge 0} \frac{\eta^k Z^k}{k!}\bigg] &= 1 + \sum_{k\ge 1} \frac{\eta^k}{k!} \E[Z^k] \underset{0 \le Z \le D}{\leq} 1 +\eta \E[Z]\sum_{k\ge 0}\frac{(\eta D)^k}{k!} %+\sum_{k\ge 2} \frac{\eta^k}{k!} \E[Z^k]
    % \\
    %&\leq  1 + \eta \E[Z]+ \frac12 \eta^2 \E[Z^2] \sum_{k\ge 2} (\eta D)^{k-2}\\ 
    \underset{\eta < 1/D}{\leq} 1 + \frac{\eta \E[Z]}{1 - \eta D}.
    \end{align*}
    Then, using $1 + x \leq \exp(x)$ and dividing by $\exp(\eta \E[Z]/(1 - \eta D))$ we get the desired result.
\end{proof}

\begin{lemma}
\label{lemma:ville_freedman}
    Let $(Z_t)$ be a process adapted to the filtration $(\mathcal F_t)$ such that $\E_{t-1}[\exp(Z_t)] \leq 1$. Then, one has for any $\delta \in (0,1)$
    \[\mathbb P\big(\exists T \ge 1 : \textstyle \sumT Z_t > \log(\delta^{-1})\big) \leq \delta.\]
\end{lemma}

\begin{corollary}
\label{cor:concentration_expectation} Let $0 \le Z_t \leq D$. Then $\tilde Z_t = \frac{Z_t}{2D}$ satisfies $\E_{t-1}[\exp(\tilde Z_t - 2 \E_{t-1}[\tilde Z_t])] \leq 1$ and with probability $1-\delta$
\[
\sumT Z_t \leq \sumT 2\E_{t-1}{Z_t} + 2D\log(\delta^{-1}).
\]
\end{corollary}

\begin{lemma}[Second order concentration]
\label{lemma:concentration_second_order}
Let $\delta \in (0,1)$. For any random variables $(Z_t) \subset \R$ adapted to $(\mathcal F_{t})$ and any $\gamma > 0$, one has with probability $1-\delta$
\[
\sumT \E_{t-1}[Z_t] \leq \sumT Z_t + \frac{\gamma}{2}\big(\E_{t-1}[Z_t^2] + Z_t^2\big) + \frac{2}{\gamma}\log(\delta^{-1}).
\]
Moreover, if $(Z_t)$ are bounded as $|Z_t| \leq D, D >0$ one has with probability $1-2\delta$
\[
\sumT \E_{t-1}[Z_t] \leq \sumT Z_t + \frac{3\gamma}{2}\E_{t-1}[Z_t^2] + \bigg(2D^2 + \frac{2}{\gamma}\bigg)\log(\delta^{-1}).
\]
\end{lemma}
\begin{proof} 
\begin{itemize} 
\item First inequality: see \citet{bercu2008exponential} or \citet[Prop. 5]{wintenberger2024stochastic}.
\item Second inequality: Since $0 \le Z_t^2 \le D$, one has for every $t \ge 1$ by Lemma~\ref{lemma:exp_inequality} on $Z_t^2$, \[
\E_{t-1}[\exp(Z_t^2/(2D^2) - 2\E_{t-1}[Z_t^2/(2D^2)])] = \E_{t-1}[\exp(Z_t^2/(2D^2) - \E_{t-1}[Z_t^2/D^2])] \leq 1.
\]
Then, by Lemma~\ref{lemma:ville_freedman} on $(Z_t^2/(2D^2) - \E[Z_t^2/D^2])$, one has
\[\sumT Z_t^2 \leq \sumT 2\E[Z_t^2] + 2D^2\log(\delta^{-1}),\] 
which gives the desired bound with a union bound.
\end{itemize}
\end{proof}

\begin{lemma}[Sup-norm of functions and predictors.]
\label{lemma:predictor_bound}
Let $B >0, f \in \B^s_{pq}(B), s > \tfrac{d}{p}$ and diameters $C_{j,k} = B2^{-j\frac{d}{2}}$. Then, both $f$ and predictors $\hat f = \sum_{j,k} c_{j,k} \varphi_{j,k}, |c_{j,k}| \leq C_{j,k}$ are bounded in sup-norm as follows:
\[
\|f\|_\infty \leq B_\infty \qquad \text{and} \qquad \|\hat f\|_\infty \le \hat B_\infty \, ,
\]
with $B_\infty = \tfrac{BM_\varphi}{1-2^{-(s-\frac{d}{p})}}, \hat B_\infty \leq (J+1)BM_\varphi$ where $M_\varphi = \|\sum_k |\phi(\cdot - k)|\|_\infty \vee \|\sum_k |\psi(\cdot - k)|\|_\infty < \infty.$
\end{lemma}
\begin{proof}
Since $f \in \B^s_{pq}(B)$ one has for every $j, \c_j = (c_{j,k})_{k \in \Lambda_j}$ 
\[
\|\c_j\|_\infty \le \|\c_j\|_p \le B2^{-j(s-\frac{d}{p} + \frac{d}{2})}.
\]
This gives
\[
\|f\|_\infty \leq \sum_{j=j_0}^{j_0+J} \|\c_j\|_\infty 2^{-j(s-\frac{d}{p})}\Big\|\sum_k |\varphi(\cdot - k)|\Big\|_\infty \leq BM_\varphi \sum_{j=j_0}^{j_0+J} 2^{-j(s - \frac{d}{p})} \leq \frac{BM_\varphi}{1 - 2^{-(s- \frac{d}{p})}}
\]
and 
\[
\|\hat f\|_\infty \le \sum_{j = j_0}^{j_0+J} |c_{j,k}|2^{j\frac{d}{2}}\Big\|\sum_k |\varphi(\cdot - k)|\Big\|_\infty \leq BM_\varphi \sum_{j=j_0}^{j_0+J} 1 \le BM_\varphi(J+1).
\]
\end{proof}

\newpage
\section{Review of multi-resolution analysis}

\label{appendix:wavelet}

In this section we present some of the basic ingredients of wavelet theory.
Let's assume we have a multivariate function $f : \R^d \to \R$.

\begin{definition}[Scaling function]
We say that a function $\phi \in L^2(\R^d)$ is the scaling function of a multiresolution analysis (MRA) if it satisfies the following conditions:
\begin{enumerate}
    \item the family \[\textstyle \{x \mapsto \phi(x - k)= \prod_{i=1}^d \phi(x_i - n_i) : k \in \mathbb Z^d\}\]
    is an ortho-normal basis, that is $\langle \phi(\cdot - k),\phi(\cdot - n) \rangle = \delta_{k,n}$;
    \item the linear spaces \[\textstyle V_0 = \big\{f = \sum_{k \in \mathbb Z^d} c_k \phi(\cdot - k), (c_k) : \sum_{k \in \mathbb Z^d} c_k^2 < \infty\big\},\dots, V_j = \{h = f(2^j \cdot) : f \in V_0\}, \dots,\]
    are nested, i.e. $V_{j-1} \subset V_j$ for all $j \geq 0$.
\end{enumerate}
\end{definition}

We note that under these two conditions, it is immediate that the functions \[\{\phi_{j,k} = 2^{dj/2}\phi(2^j \cdot-k)\, , k \in \mathbb Z^d\}\] form an ortho-normal basis of the space $V_j, j \in \mathbb N$.
One can define the projection kernel of $f$ over $V_j$ (from here we also say kernel projection at scale or level $j$) as \begin{equation}
\label{eq:proj_kernel}
    K_j f(x) := \sum_{k \in \mathbb Z^d} \langle f , \phi_{j,k} \rangle \phi_{j,k}(x) = \int_{\R^d} K_j(x,y) f(y) \, dy\,,
\end{equation}
with $K_j(x,y) = \sum_{n \in \mathbb Z^d}  \phi_{j,k}(x)\phi_{j,k}(y) = \sum_{k\in \mathbb Z^d} 2^{dj}\phi(2^jx-n)\phi(2^jy-n)$ (which is not of convolution type) but has comparable approximation properties that we detail after. 

\paragraph{Incremental construction via wavelets.}
Since the spaces $(V_j)$ are nested, one can define nontrivial subspaces as the orthogonal complements $W_j := V_{j+1} \ominus V_j$. We can then telescope these orthogonal complements to see that each space $V_j, j \geq j_0$ can be written as
\[
V_j = V_{j_0} \oplus \bigg(\bigoplus_{l=j_0}^{j} W_l\bigg) \quad \text{for any } j_0 \in \mathbb{N}.
\]
Let $\psi$ be a mother wavelet corresponding to the scaling function $\phi$. The associated wavelets are defined as follows: for $E = \{0,1\}^d \setminus \{0\}$, we set
\[
\psi^\epsilon(x) = \psi^{\epsilon_1}(x_1) \cdots \psi^{\epsilon_d}(x_d), \quad 
\psi^\epsilon_{j,n} = 2^{jd/2} \psi^\epsilon(2^j x - n), \quad 
j \geq 0, \quad n \in \mathbb{Z}^d,
\]
where $\psi^0 = \phi$, $\psi^1 = \psi$. For each $j$, these functions form an orthonormal basis of $W_j$.

Analogously, one can now observe that for every $j \geq j_0$,
\begin{equation} 
\label{eq:telescopage}
K_j f = K_{j_0} f + \sum_{l=j_0}^{j-1} \left(K_{l+1} f - K_l f\right),
\end{equation}
where each increment in the sum can be written as
\[
K_{j+1} f - K_j f = \sum_k \sum_\epsilon \langle f, \psi^\epsilon_{j,k} \rangle \psi^\epsilon_{j,k},
\]
where for each $j \geq 1$, the set
\[
\left\{ \psi^e_{j,k} = 2^{dj/2} \psi^\epsilon(2^j x - k) : \epsilon \in E,\, k \in \mathbb{Z}^d \right\}
\]
forms a basis of $W_j$ for some wavelet $\psi$, with $E := \{0,1\}^d \setminus \{0\}$. For simplicity, we include the index $\epsilon$ in the multi-index $k$. Finally, the set $\{\phi_{j_0,k}, \psi_{j,k}\}$ constitutes a \emph{wavelet basis}.

For our results we will not be needing a particular wavelet basis, but any that satisfies the following key properties.

\begin{definition}[$S$-regular wavelet basis]
\label{def:regular_wavelet}
Let $S \in \mathbb{N}^*$ and $j_0 = 0$. The multiresolution wavelet basis \[\{\phi_{k}=\phi(\cdot - k), \psi_{j,k} = 2^{jd/2}\psi(2^d\cdot - k)\}\] of $L^2(\R^d)$ with associated projection kernel $K(x,y) = \sum_k \phi_k(x)\phi_k(y)$ is said to be $S$-regular if the following conditions are satisfied:
\begin{enumerate}[label=(D.\arabic*),ref=D.\arabic*,noitemsep,topsep=-\parskip]
    \item \label{def:s_regular_wavelet:D1} \textbf{Vanishing moments and normalization:}
    \[
    \textstyle \int_{\mathbb{R}^d} \psi(x)\, x^\alpha\, dx = 0 \quad \text{for all multi-indices } \alpha \text{ with } |\alpha| < S,
    \qquad \int_{\mathbb{R}^d} \phi(x)\, dx = 1.
    \]
    Moreover, for all $v \in \mathbb{R}^d$ and $\alpha$ with $1 \le |\alpha| < S$,
    \[
    \textstyle \int_{\mathbb{R}^d} K(v, v + u)\, du = 1, \quad
    \int_{\mathbb{R}^d} K(v, v + u)\, u^\alpha\, du = 0.
    \]

    \item \label{def:s_regular_wavelet:D2} \textbf{Bounded basis sums:}
    \[
    \textstyle M_\phi := \sup_{x \in \mathbb{R}^d} \sum_{k} |\phi(x - k)| < \infty, \qquad
    M_\psi := \sup_{x \in \mathbb{R}^d} \sum_{k} |\psi(x - k)| < \infty .
    \]

    \item \label{def:s_regular_wavelet:D3} \textbf{Kernel decay:} For $\kappa(x,y)$ equal to $K(x,y)$ or $\sum_{k} \psi(x - k) \psi(y - k)$, there exist constants $c_1, c_2 > 0$ and a bounded integrable function $\phi : [0,\infty) \to \mathbb{R}$ such that
    \[
    \textstyle \sup_{v \in \mathbb{R}^d} |\kappa(v, v - u)| \le c_1 \phi(c_2 \|u\|), \qquad C_S := \int_{\mathbb{R}^d} \|u\|^S \phi(\|u\|)\, du < \infty.
    \]
\end{enumerate}
\end{definition}

% \begin{definition}[$S$-regular wavelet]
% \label{def:regular-wavelet}
% A multiresolution wavelet basis \[\{ \Phi_k, \Psi^\epsilon_{j,k} : k \in \mathbb Z^d, j \in \mathbb N, \epsilon \in \{0,1\}^d \setminus \{0\} \}\] with projection kernel $K(x,y) = \sum_k \Phi(x-k)\Phi(y-k)$ is said to be $S$-regular for some $S \in \mathbb N^*$ if the following conditions are satisfied: \begin{enumerate}
%     \item $\int_{\R^d} \Psi(u) u^l du = 0, \forall l = 0, 1, \dots, S-1, \quad \int_{\R^d} \Phi(x) dx = 1$, and for all $v \in \R^d$, and multi-index $s \in \mathbb N^d$ such that $|s| = 1, \dots, S-1$,
%     \[
%     \int_{\R^d} K(v,v+u)\, du = 1 \quad \text{and} \quad \int_{\R^d} K(v,v+u) u^s \, du= 0\, ;
%     \]
%     \item $\sum_k |\Phi(\cdot - k)| \in L^\infty(\R^d)$ and $\sum_k |\Psi(\cdot - k)| \in L^\infty(\R^d)$;
%     \item $\sup_{v \in \R^d} |K(v,v-u)| \leq c_1g(c_2\|u\|)$ for $c_1,c_2 >0$ and every $u \in \R^d$, for some bounded integrable function $g$ such that $\int \|u\|^S g(\|u\|) du < \infty$.
% \end{enumerate}
% \end{definition}

\paragraph{Case of a bounded compact $\X \subset \R^d$.}
The above definition applies to wavelet systems on $\R^d$, but can be extended to compact domains $\X \subset \mathbb{R}^d$ using standard boundary-corrected or periodized constructions. Notable examples include the compactly supported orthonormal wavelets of \citet[Chapter~7]{daubechies1992ten} and the biorthogonal, symmetric, and highly regular wavelet bases of \citet{cohen1992biorthogonal}.
At the $j$-th level, there are now $O(2^{jd})$ wavelets $\psi_{j,k}$, which we index by $k \in \Lambda_j$, the set of indices corresponding to wavelets at level $j$. This coincides with the expansion used in Equation~\eqref{eq:wavelet_decomposition}.